\documentclass[letterpaper,onefignum,onetabnum]{siamart220329}
\usepackage{lipsum}
\usepackage{amssymb,amsfonts}
\usepackage{amsmath,mathrsfs}
\usepackage{graphicx,subfigure}
\usepackage{epstopdf}
\usepackage{algorithm,algorithmic}
\usepackage{color}
\usepackage{bm}
\usepackage{booktabs,array}
\usepackage{multirow,multicol}
\usepackage{rotating}
\usepackage{geometry}
\usepackage{nameref}
\allowdisplaybreaks
\ifpdf
\DeclareGraphicsExtensions{.eps,.pdf,.png,.jpg}
\else
\DeclareGraphicsExtensions{.eps}
\fi

\renewcommand{\thefigure}{\Roman{figure}}

\newcommand{\adots}{\reflectbox{\ensuremath{\ddots}}}

\newsiamremark{remark}{Remark}
\newsiamremark{hypothesis}{Hypothesis}
\crefname{hypothesis}{Hypothesis}{Hypotheses}
\newsiamthm{claim}{Claim}
\newsiamremark{example}{Example}

\headers{$\tau$ preconditioner for Riesz fractional diffusion equations}{Z.-H. She, X. Zhang, X.-M. Gu, et al.}

\title{A $\tau$-preconditioner for quasi-compact difference schemes solving variable-coefficient Riesz fractional diffusion equations
        }

\author{Zi-Hang She\thanks{Department of Mathematics, Hanshan Normal University, Chaozhou, Guangdong 521041, P.R. China (\email{zhshe@hstc.edu.cn})}
	\and Xue Zhang\thanks{School of Mathematics,
		Southwestern University of Finance and Economics,
		Chengdu, Sichuan 611130, P.R. China (\email{xuzhangym@smail.swufe.edu.cn})}
    \and Xian-Ming Gu\thanks{Corresponding author. School of Mathematics, Southwestern University of Finance and Economics,
		Chengdu, Sichuan 611130, P.R. China
	(\email{guxianming@live.cn})}
	\and Stefano Serra-Capizzano \thanks{Department of Science and High Technology,
		University of Insubria, Como 22100, Italy and Department of Information Technology,
		University of Uppsala, Uppsala SE-751 05, Sweden (\email{stefano.serrac@uninsubria.it}, \email{stefano.serra@it.uu.se})}.
}

\usepackage{amsopn}

\ifpdf
\hypersetup{
  pdftitle={DLRA for strongly damped wave equations},
  pdfauthor={Y.-L. Zhao}
}
\fi

\begin{document}

\maketitle

\begin{abstract}
In the present study, we consider the preconditioned generalized minimal residual (GMRES) method for the asymmetric linear systems arising from the $d$-dimensional Riesz space fractional diffusion equations (RSFDEs). The Crank-Nicolson scheme and a quasi-compact finite difference method are used to discretize the temporal derivative and Riesz space fractional derivatives in such RSFDEs, respectively. For the $d$-dimensional discretized RSFDEs, the corresponding coefficient matrix is the sum of a product of a $d$-level block tridiagonal matrix multiplying a diagonal matrix and a $d$-level Toeplitz matrix. We develop a sine transform based preconditioner (namely $\tau$ preconditioner) to accelerate the convergence of the GMRES method. Theoretical analysis shows that the upper bound of relative residual norm of the GMRES method with the proposed preconditioner is mesh-independent, which leads to a linear convergence rate. Numerical results are presented to confirm the theoretical results regarding the preconditioned matrix and to illustrate the efficiency of the proposed preconditioner.
\end{abstract}

\begin{keywords}
 GMRES method; Riesz space fractional diffusion equation; Numerical range; Toeplitz matrix; $\tau$-preconditioner
\end{keywords}

\begin{AMS}
34A08; 65F08; 65F10
\end{AMS}

\section{Introduction}\label{introduction}

In recent years, numerous fractional derivatives have emerged and been used to develop mathematical models for a wide range of real-world systems, characterized by memory, historical, or non-local effects, such as biology (species-competing model) \cite{Owolabi}, medical science (epidemic model) \cite{Parra}, finance (option-pricing models) \cite{ZhangH2016}, hydrology (contaminant transport model in porous media) \cite{Aghdam}. 

In the present work, we focus on developing efficient numerical methods for the following $d$-dimensional Riesz space fractional diffusion equation (RSFDE) with variable coefficients:
\begin{equation}
\label{RSFDE}
\left\{
\begin{aligned}
&e(\bm{x},t) \partial_t u(\bm{x},t)= \sum\limits_{i=1}^{d} \kappa_i \partial^{\alpha_i}_{x_i} u(\bm{x},t)+f(\bm{x},t),\quad  (\bm{x},t)\in \Omega \times(0,T],\\
&u(\bm{x},t)=0,  \hspace{5.4cm} \bm{x} \in \partial \Omega,\ t\in(0,T],\\
&u(\bm{x},0)=\psi(\bm{x}),   \hspace{4.8cm} \bm{x} \in \bar{\Omega}, \\
\end{aligned}
\right.
\end{equation}
where $1<\alpha_i<2$, $1 \leq i \leq d$, $\Omega = \prod_{i=1}^{d} (L_i,R_i)$ is the spatial region, $\partial \Omega$ is the boundary of $\Omega$, $\bar{\Omega} = \Omega \cup \partial \Omega$, $\kappa_i\geq 0$ are the diffusion coefficients satisfying 
the assumption $\sum^{d}_{i=1}\kappa_i\neq 0$, $\bm{x} = (x_1,x_2,\cdots,x_d) \in \mathbb{R}^{d}$, $u(\bm{x},t)$ is the unknown function to be found, $f(\bm{x},t)$ is the source term, $e(\bm{x},t)$ is a bounded positive-valued function with positive minimum, and $\psi(\bm{x})$ is the initial condition. For $\alpha_i \in (1,2)$, the Riesz space fractional derivative of order $\alpha_i$ for a function defined on $[L_i,R_i] \times [0,T]$, is defined by

\[ \partial^{\alpha_i}_{x_i} u(\bm{x},t)=\sigma_{\alpha_i}\left({_{L_i}}D_{x_i}^{\alpha_i} +{_{x_i}}D_{R_i}^{\alpha_i}\right)u(\bm{x},t),~~(\bm{x},t)\in \Omega \times [0,T],\]
where $\sigma_{\alpha_i}=-\dfrac{1}{2\cos(\alpha_i\pi/2)}$, and ${_{L_i}}D_{x_i}^{\alpha_i}$ and ${_{x_i}}D_{R_i}^{\alpha_i}$ are the left- and right-sided Riemann-Liouville (RL) spatial fractional derivatives defined, respectively, by
\begin{equation*}
\begin{aligned}
{_{L_i}}D_{x_i}^{\alpha_i}u(\bm{x},t)  = & \frac{1}{\Gamma(2-\alpha_i)} \frac{\partial^2}{\partial x_i^2}\int_{L_i}^{x_i}
  \frac{u(x_1,\cdots,x_{i-1},\xi,x_{i+1},\cdots, x_{d},t)}{(x_i-\xi)^{\alpha_i-1}}d\xi, \\
{_{x_i}}D_{R_i}^{\alpha_i}u(\bm{x},t) = & \frac{1}{\Gamma(2-\alpha_i)}\frac{\partial^2}{\partial x_i^2}\int_{x_i}^{R_i}
  \frac{u(x_1,\cdots,x_{i-1},\xi,x_{i+1},\cdots, x_{d},t)}{(\xi-x_i)^{\alpha_i-1}}d\xi,
\end{aligned}
\end{equation*}
where $\Gamma(\cdot)$ is the Gamma function. The local properties of traditional partial derivatives are not always suitable for simulating the process of groundwater solute transport, in the sense that they do not consider their spatiotemporal correlation. The non-locality of fractional derivative or fractional order proves more effective in describing the anomalous diffusion phenomenon in this process. Eq. (\ref{RSFDE}) is mainly considered to model the absorbing solute transport \cite{HuangG2006}, where $u$ represents the resident solute concentration and the variable coefficient $e(\bm{x},t)$ is a retardation contribution which is affected by factors such as the heat conduction \cite{Paraiba}. For other real-world applications of Eq. (\ref{RSFDE}), see \cite{ZhangXue2025}.

However, for both integer and fractional orders, the exact solutions of partial differential equations (PDEs) are rarely explicitly known and when it happens often the related representation formulae are not computationally efficient. This property is even stronger when dealing with the more involved fractional PDEs, leading to the emergence of numerous numerical methods over the past decade. Furthermore, while the locality of PDEs leads to a band multilevel structure when a local numerical method is used, the nonlocality of the fractional-order operators implies that dense matrices are encountered even when the approximation method is of local type, as finite differences, finite volumes, finite elements etc.

As for discretization methods in connection with fractional PDEs, we can refer to finite differences (see \cite{Meerschaert2006}), finite elements (see \cite{WangH2013}), finite volumes (see \cite{ZhangX2005}),  basis function configuration methods (see \cite{Sousa2015}). The list could be indeed very long but we stop here since the main focus of the present contribution relies on fast solvers for the related linear systems, by using iterative methods and spectral analysis tools for understanding the convergence features of our preconditioned Krylov subspace solvers.

Fast algorithms for space fractional PDEs are currently popular. A recent study has revealed that the banded preconditioner with diagonal compensation \cite{She2021} can reduce the condition number of the preconditioned matrix. Moreover, Lei and Sun proposed the preconditioned CGNR method with a circulant preconditioner to solve such Toeplitz-like systems, incurring a computational cost of $\mathcal{O}(N_s \log N_s)$, where $N_s$ is the number of total grid nodes \cite{Lei2013}. In addition, for the discretization and fast algorithms of some PDEs with Riesz fractional derivatives, there are other methods for readers to refer to \cite{Breiten2016,Zhu2021}. 

The $\tau$-algebra-based preconditioning strategies have attracted significant attention in both recent and earlier literature, due to the fact that the $\tau$ matrix can be diagonalized by the discrete sine matrix. The definition of the related matrix-algebra and the initial computational results, such as those for accurate eigenvalue solvers, date back to the work of Bini and Capovani \cite{BC-LAA-83} in 1983. Subsequently, $\tau$-preconditioners emerged as a superior alternative for the fast solution of real symmetric Toeplitz linear systems \cite{Bini90,DiBe-SISC-95,DS-NumMath-99,SSC-MC-97,Serra1}. 

Furthermore, the analysis of other techniques, such as band-Toeplitz preconditioning and multigrid methods, originated from the properties and studies of the $\tau$-algebra \cite{Chan-IMA-91,CCS-SISC-98,FS-SISC-96,SSC-MC-97}. For a comprehensive overview of several iterative methods for quickly solving structured linear systems in the late 90s, refer to \cite{CN-SIREV-96,DS-NumMath-99}.

Theoretical studies regarding the rank correction and the role of the generating function have indicated that for ill-conditioned real symmetric Toeplitz systems, the $\tau$-algebra exhibits superior preconditioning properties and numerical performance compared to the circulant preconditioner \cite{SSC-SIMAX-99}.  
This aspect was also investigated in the context of multilevel Toeplitz structures \cite{NSV-TCS-04,SSC-LAA-02} and references therein. The main message is that in general any matrix-algebra with unitary transform fails to achieve clustering and essential spectral equivalence when dealing with high-dimensional problems and this was also confirmed in the circulant setting for fractional-order problems \cite{LeiS2016} recently. However, when the order of the zero of the generating function is not high, the $\tau$-preconditioning can achieve essential spectral equivalence and hence optimality \cite{Noutsos}.

In recent years, a significant contribution to the use of $\tau$-algebra based preconditioners and multigrid method was given in \cite{DMS-JCP-16,DMS-SISC-18}.
Notably, in contrast to banded preconditioners, the $\tau$-preconditioner offers substantial computational advantages and faster iteration speeds in high-dimensional cases. 
Barakitis et al. \cite{BES-NLA-22} and Huang et al. \cite{HuangX2022} investigated the distribution of eigenvalues in the $\tau$-preconditioned matrix. Their work was further extended to multilevel Toeplitz matrices generated by functions with a unique fractional-order zero at zero. Zeng et al. proposed a novel approximate inverse preconditioner for diagonal-plus-Toeplitz coefficient matrices in RSFDE. They employed piecewise interpolation polynomials and $\tau$-matrix approximation in their approach \cite{Zeng2022}.
Lu et al. \cite{Lu2021} put forward a splitting preconditioner founded on sine transform for Toeplitz-like linear systems in one- and two-dimensional RSFDEs. This preconditioner demands a computational complexity of $\mathcal{O}(N_s\log N_s)$ operations at each time level. Li and Hon \cite{Lisimax25} focused on the multilevel Toeplitz systems discretized from RL space fractional diffusion equations. They transformed such multilevel Toeplitz systems into multilevel Hankel systems and thereby proposed a multilevel $\tau$-preconditioner to accelerate the convergence rate of the MINRES method. Meanwhile, they theoretically proved that this convergence rate is mesh-independent.
The authors of the study \cite{HuangLei2023} deliberated on a novel fourth-order difference scheme for two-dimensional RSFDE. They resolved this scheme using the PCG method underpinned by $\tau$-matrix preconditioners.
Regarding other efficient preconditioners, we direct readers to the work in \cite{Aceto2023}. In this reference, the proposed preconditioner is constructed upon a rational approximation of the Riesz fractional operator. Crucially, since the release of the preliminary version of our present work [She et al., \href{https://arxiv.org/abs/2404.10221}{arXiv:2404.10221}], the proposed $\tau$-preconditioning strategy, specifically its numerical range-based analytical framework for the convergence rate of preconditioned GMRES method, has served as a foundational basis for several follow-up works. These include high-order schemes for variable-coefficient RSFDEs \cite{HUANG2026}, finite volume discretizations of conservative space-fractional PDEs \cite{Qu2026x}, and all-at-once linear systems for evolutionary PDEs \cite{Huang2025x,Huang25y}.

It is worth noting that for fractional PDEs with variable coefficients, the discrete linear systems are often asymmetric. As a result, even first-order finite difference schemes, which usually yield well-structured matrices, have not yet been thoroughly analyzed in terms of the convergence rate of preconditioned GMRES methods under mild or no assumptions. 
Recently, Lin et al. \cite{LinXLNG2024} proposed a $\tau$-preconditioned GMRES method for the discrete linear systems arising from first- or second-order finite difference methods of RSFDEs. Under the assumptions that the diffusion coefficient is Lipschitz continuous and the time step is sufficiently small, they also derived an explicit convergence-rate estimate for the preconditioned iteration. Compared with first-order schemes, the high order schemes for fractional operators do not increase the computational cost but greatly improve the accuracy. The reason is that the stencils for high order discretization schemes and low order ones are the same \cite{TZD1}. On this basis, we first propose a class of fourth order quasi-compact schemes for Eq. (\ref{RSFDE}) in this article, as discussed in \cite{Zhang24}. The coefficient matrix of the discretized linear system is composed of a $d$-level Toeplitz matrix and an asymmetric block tridiagonal matrix. Owing to its asymmetry and ill-conditioning, employing the Krylov subspace method to solve this system leads to a slow convergence rate.

One of the contributions of this paper is the introduction of an asymmetric preconditioning method grounded in the fast discrete sine transform. Additionally, an analysis of the GMRES convergence rate is conducted. The analysis reveals that the convergence rate of the proposed preconditioned GMRES (PGMRES) method is independent of the derivative order, spatial dimension, and discretization step size.
To the best of our knowledge, this could potentially be the first PGMRES method for variable-coefficient fractional diffusion equations that has undergone rigorous theoretical analysis, so going beyond the analysis in the pure multilevel Toeplitz setting. Furthermore, our convergence analysis imposes no additional assumptions beyond the basic well-posedness of the problem; in particular, neither the smoothness of the diffusion coefficient nor any restriction on the time-step size is required.

The subsequent sections of this paper are structured as follows. In Section \ref{Section:MD}, a compact numerical scheme for the multi-dimensional problem of Eq. (\ref{RSFDE}) is presented. This scheme encompasses the multi-dimensional discretized linear system, along with a detailed derivation of the right-hand side vector based on the source term $f$. Section \ref{Section:tau} introduces the sine transform-based $\tau$-preconditioner. Additionally, we put forward a two-sided preconditioner for the discretized system. Meanwhile, we prove the convergence rate of the PGMRES method.
In Section \ref{Numer}, a set of carefully selected numerical experiments are conducted. These experiments aim to assess the accuracy of the approximation methods and the performance of the proposed fast algorithm.
The paper ends with Section \ref{sec-end}, which summarizes the main conclusions, followed by \nameref{appA} and \nameref{appB}.

\section{Numerical scheme of multi-dimensional cases}\label{Section:MD}
In the present section, we provide an implicit numerical discretization under uniform grids for solving the RSFDE (\ref{RSFDE}). For positive integers $N$ and $M$, let $\Delta t=T/M$ and $h_i=(R_i-L_i)/(N+1)$. Define the temporal grid and the spatial grid, respectively, by $\{t_m|t_m=m\Delta t,~0\leq m\leq M\}$  and $\{x_{i,j}|x_{i,j}=L_i+jh_i,~1\leq i\leq d,~1\leq j\leq N\}$. For convenience, we define a multi-index $\mathbb{J} = (j_1,j_2,\cdots,j_d)$ and the ranges $\Upsilon = \{j| 1 \leq j \leq N\}$, $\tilde{\Upsilon} = \{j| 0 \leq j \leq N+1\}$, 
\begin{equation}
\mathcal{K}=\Upsilon \times \underbrace{\cdots}_{d-2} \times \Upsilon,\quad \tilde{\mathcal{K}} = \tilde{\Upsilon} \times \underbrace{\cdots}_{d-2} \times \tilde{\Upsilon}.
\notag
\end{equation}
 We apply the Crank-Nicolson (CN) technique to discretize the temporal derivative $\partial_t u(\bm{x},t)$, and apply the fractional centered difference (FCD) formula \cite{Zhang24} to discretize the Riesz space fractional derivative $\partial^{\alpha_i}_{x_i} u(\bm{x},t)$. 

 At the current stage, we extend the FCD formula on the Riesz derivative at the point $\bm{x}_{\mathbb{J}}$ for Eq. (\ref{RSFDE}) and the Crank-Nicolson method at the time point $t_{m+1/2}$:
\begin{equation}
\label{eq2.1}
\begin{aligned}
		 &e_{\mathbb{J}}^{m+1/2} \delta_t u_{\mathbb{J}}^{m+1/2} \approx \frac{1}{2} \sum\limits_{i=1}^{d} \kappa_i \delta^{\alpha_i}_{x_i} \left( u_{\mathbb{J}}^{m+1} + u_{\mathbb{J}}^{m}\right) + f_{\mathbb{J}}^{m+1/2},\ \mathbb{J} \in \mathcal{K},\ 0 \leq m \leq M-1, 
	\end{aligned}
\end{equation}
where $u_{\mathbb{J}}^m$ is the numerical solution of $u\left({\bm x}_{\mathbb{J}}, t_m \right)$, for $\mathbb{J} \in \tilde{\mathcal{K}}$, and $\delta^{\alpha_i}_{x_i}$ is a linear discrete operator of the FCD method, such that 
$$
\left(\partial^{\alpha_i}_{x_i}u\right)\left(\bm{x}_{\mathbb{J}}, t_m\right) \approx \delta^{\alpha_i}_{x_i} u_{\mathbb{J}}^m
=-\frac{1}{h_i^{\alpha_i}}\sum_{j=-\left[\frac{R_i-x_i}{h_i}\right]}^{\left[\frac{x_i-L_i}{h_i}\right]}s_{j}^{(\alpha_i)}
u(x_1,\cdots,x_{i-1},x_i-jh_i,x_{i+1},\cdots,x_d,t_m)
$$
given by \cite{Celik}. The FCD weights of the Riesz fractional derivative of order $\alpha$ are defined as 
$$
s_{j}^{(\alpha)}=\frac{(-1)^j\Gamma(\alpha+1)}{\Gamma(\alpha/2-j+1)\Gamma(\alpha/2+j+1)},~~j=0,\pm 1,\pm 2,\cdots.
$$

With reference to \cite{Sun2016,Zhang24}, we apply the fourth-order compact operator $\mathscr{H}$ in \eqref{eq2.1}, i.e.,
\begin{equation}
\label{eq2.2}
\begin{aligned}
		 &\mathscr{H} e_{\mathbb{J}}^{m+1/2} \delta_t u_{\mathbb{J}}^{m+1/2} = \frac{1}{2} \sum\limits_{i=1}^{d} \kappa_i \mathscr{H}_i \delta^{\alpha_i}_{x_i}\left( u_{\mathbb{J}}^{m+1} + u_{\mathbb{J}}^{m}\right)+\mathscr{H} f_{\mathbb{J}}^{m+1/2},\ \mathbb{J} \in \mathcal{K},\ 0 \leq m \leq M-1. 
	\end{aligned}
\end{equation}
 where 
 \begin{equation*}
 \begin{aligned}
  &  \mathscr{H}_{\alpha_l} u_{\mathbb{J}(l)}^m=\begin{cases}
   \frac{\alpha_l}{24} u_{\mathbb{J}(l)-1}^m + \left(1-\frac{\alpha_l}{12}\right) u_{\mathbb{J}(l)}^m + \frac{\alpha_l}{24} u_{\mathbb{J}(l)+1}^m,& \mathbb{J}(l) \in \Upsilon, \\
   u_{\mathbb{J}(l)}^m, &\mathbb{J}(l) \in \tilde{\Upsilon} \setminus \Upsilon, 
  \end{cases} \\
  & (\text{Here we only focus on the coordinate}\ \mathbb{J}(l)=j_l \ \text{acted by} \ \mathscr{H}_{\alpha_l} \ \cite{Sun2016})
 \end{aligned}
\end{equation*}
 $\mathscr{H}=\prod\limits_{l=1}^d \mathscr{H}_{\alpha_l}$, $\mathscr{H}_i=\prod\limits_{l=1,l \neq i}^d \mathscr{H}_{\alpha_l}$ and $u_{\mathbb{J}}^{m}$ denotes the grid approximation of $U_{\mathbb{J}}^{m}$. Clearly, the initial value is written as $u_{\mathbb{J}}^{0} = \psi(\bm{x}_{\mathbb{J}})$ for all $\mathbb{J} \in \tilde{\mathcal{K}}$. In fact, it is worth noting that two numerical schemes \eqref{eq2.1}--\eqref{eq2.2} are extended from our recent work \cite{Zhang24}, where the stability and convergence analysis of these schemes are given under $d = 1,2$ (the corresponding extensions to $d\geq 3$ follow analogously using techniques from \cite{ZhangXue2025,Zhang24}).

 The uniform spatial grid is denoted as $\mathbb{G} = \{ \bm{x}_{\mathbb{J}} | \mathbb{J} \in \mathcal{K}\}$. Suppose that the grid $\mathbb{G}$ is endowed with lexicographic ordering. Subsequently, we derive the discretization of the considered RSFDE as follows.

\begin{equation}
\label{tls}
\begin{aligned}
A_{\alpha}^{(m)}\mathbf{u}^{(m+1)}&:=
\left(B_{\alpha}E^{(m+1/2)} + S_{\alpha}\right)\mathbf{u}^{(m+1)}\\
&=\left(B_{\alpha}E^{(m+1/2)} - S_{\alpha}\right)\mathbf{u}^{(m)}
+\Delta t \mathbf{F}^{(m+1/2)}\\
&:=\mathbf{b}^{(m)},~~0\leq m\leq M-1,
\end{aligned}
\end{equation} 
where $\mathbf{u}^{(0)}=\psi(\mathbb{G})$, $E^{(k+1/2)}=$diag$(e(\mathbb{G},t_{k+1/2}))$,

\begin{equation}
\label{Salpha}
 S_{\alpha} = \eta_1(B_{\alpha_d}\otimes \cdots \otimes B_{\alpha_2} \otimes S_{\alpha_1}) + \cdots + \eta_d(S_{\alpha_d} \otimes B_{\alpha_{d-1}} \otimes \cdots \otimes B_{\alpha_1})
\end{equation}
with $\eta_i=\frac{\kappa_i \Delta t}{2h_i^{\alpha_i}}$ for $1 \leq i \leq d$. The Kronecker product, denoted by ``$\otimes$'' in \eqref{Salpha}, is an operation on two matrices of arbitrary size that produces a block matrix. It satisfies the following properties:
$$
(A \otimes B)^\top = A^\top \otimes B^\top,~(A \otimes B)(C \otimes D) = (AC) \otimes (BD)
$$
for any matrices $A$, $B$, $C$, and $D$ of suitable sizes. 
The matrix $B_{\alpha}$ is defined by 
$$ B_{\alpha} = B_{\alpha_d} \otimes B_{\alpha_{d-1}} \otimes \cdots \otimes B_{\alpha_1}$$ 
with $B_{\alpha_i}$ being a tridiagonal matrix for each $1 \leq i \leq d$, given by
\begin{equation}
\label{Ha}
B_{\alpha_i}=I_{N} + \frac{\alpha_i}{24}\mathrm{tridiag(1,-2,1)}\in \mathbb{R}^{N\times N}.
\end{equation}
Furthermore, $S_{\alpha_i}$ is a symmetric dense Toeplitz matrix, for each $1 \leq i \leq d$, having the expression
\begin{equation}
\label{eq0}
S_{\alpha_i}=
    \begin{bmatrix}
      s_{0}^{(\alpha_i)} & s_{1}^{(\alpha_i)} & \ldots & s_{N-2}^{(\alpha_i)} & s_{N-1}^{(\alpha_i)} \\
      s_{1}^{(\alpha_i)} & s_{0}^{(\alpha_i)} & s_{1}^{(\alpha_i)} & \ldots & s_{N-2}^{(\alpha_i)}\\
     \vdots & \ddots & \ddots & \ddots& \vdots \\
      s_{N-2}^{(\alpha_i)}& \ddots & \ddots & \ddots & s_{1}^{(\alpha_i)}\\
      s_{N-1}^{(\alpha_i)} & s_{N-2}^{(\alpha_i)}&\ldots & s_{1}^{(\alpha_i)} &s_{0}^{(\alpha_i)}
    \end{bmatrix},
\end{equation}
where
$$
s_0^{(\alpha_i)}=\frac{\Gamma(\alpha_i+1)}{\Gamma^2(\alpha_i/2+1)},~~s_k^{(\alpha_i)}=
\left(1-\frac{\alpha_i+1}{\alpha_i/2+k}\right)s_{k-1}^{(\alpha_i)},\quad k=1,2,\cdots,N-1.
$$
We note that a  matrix $T$ of size $N \times N$ is called Toeplitz matrix, if $[T]_{ij} = t_{i-j}$, $i,j=1,2,\cdots,N$.

It is readily verified that the coefficients $\{s_{k}^{(\alpha_i)}\}_{k\geq 0}$ satisfy
\begin{equation}
 \label{As1}
s_{0}^{(\alpha_i)}>0,~~s_{k}^{(\alpha_i)}< 0,~~ k=1,2,\cdots, ~~\mbox{and} ~~
s_0^{(\alpha_i)}+2\sum_{k=1}^\infty s_{k}^{(\alpha_i)}=0,
\end{equation}
which implies that $S_{\alpha_i}$ is symmetric positive definite. The same conclusion can also be found in \cite{Sun2016}.

The vector $\mathbf{F}^{(m+1/2)}$ depends on the spatial dimension. For clarity, its form is detailed for $d=1,2,3$ in \nameref{appA}. \nameref{appB} further presents the full derivation of scheme \eqref{eq2.2} for the one-dimensional RSFDE \eqref{RSFDE}, together with the resulting matrix–vector formulation. Derivations for higher dimensions follow analogously. For instance, one can refer to the derivation process of the two-dimensional case as described in Ref. \cite{ZhangXue2025}. 

In order to solve the discretized linear system \eqref{tls} efficiently, the iterative method with fast preconditioning strategy is introduced in the next section.

\section{Sine transform based preconditioner}
\label{Section:tau}
Before defining the preconditioner, we first introduce some notations to be used later on. For a symmetric Toeplitz matrix, $T\in\mathbb{R}^{N\times N}$, with $[t_0,t_1,\cdots,t_{N-1}]^\top\in\mathbb{R}^{N}$ being the first column, define its natural $\tau$-matrix approximation as \cite{BC-LAA-83,Bini90,Gu2021,Serra1}
\begin{equation}
\label{tau}
\tau(T)=T-H(T),
\end{equation}
where $H(T)$ is a Hankel matrix defined as follows,
\begin{equation}
H(T)=
    \begin{bmatrix}
      t_2& t_3& \cdots & t_{N-1}&0 &0 \\
      t_3& \adots & \adots  &\adots &\adots &0\\
       \vdots&\adots &\adots &\adots &\adots&t_{N-1}\\
       t_{N-1}&\adots &\adots& \adots & \adots&\vdots\\
      0&\adots&\adots&\adots&\adots&t_3\\
      0&0&t_{N-1}&\cdots&t_3&t_2
    \end{bmatrix}.
\end{equation}
Especially, if the Toeplitz matrix $T$ is tridiagonal, then it is just a $\tau$-matrix.

 The $\tau$-matrix defined in (\ref{tau}) can be diagonalized by the sine transform matrix, i.e.,
 \begin{equation}
 \label{diag}
 \tau(T)=SQS,
 \end{equation}
 where $Q=[\mathrm{diag}(q_i)]_{i=0}^{N-1}$ is a diagonal matrix with
 \begin{equation}
 \label{q}
 q_i=t_0+2\sum\limits_{j=1}^{N-1}t_j\cos \left(\frac{\pi j (i+1)}{N+1}\right),~~i= 0,1,\cdots,N-1,
 \end{equation}
 and $S$ is a sine transform matrix defined by
 $$
 S=\sqrt{\frac{2}{N+1}}\left[\sin\left(\frac{\pi j k}{N+1}\right)\right]_{j,k=1}^{N}.
 $$
 Using the fast sine transform (FST), the implementation of $S\mathbf{r}$ for any given vector $\mathbf{r}$ requires only $\mathcal{O}(N\log N)$ real operations and $\mathcal{O}(N)$ storage. Thus, both the diagonal entries $\{q_i\}_{i=0}^{N-1}$ in (\ref{q}) and the implementation  of $\tau(T)^{-1}\mathbf{r}$ (or $\tau(T)^{-1/2}\mathbf{r}$) only requires $\mathcal{O}(N\log N)$ real operations.
The fact that the required operations are only real represents a further advantage with respect to the celebrated fast Fourier transforms and the use of the circulant algebra.

Evidently, for $1 \leq i \leq d$, the matrix $B_{\alpha_i}$ defined in (\ref{Ha}) is a $\tau$-matrix. Owing to its strict diagonal dominance, this matrix is invertible. Instead of directly solving the linear systems (\ref{tls}), we can solve the following equivalent form.
\begin{equation}
\label{equi}
\begin{aligned}
\tilde{A}_{\alpha}^{(m)}\mathbf{u}^{(m+1)}=\tilde{\mathbf{b}}^{(m)},~~0\leq m\leq M-1,
\end{aligned}
\end{equation}
where
\begin{equation*}
\tilde{A}_{\alpha}^{(m)}=E^{(m+1/2)}  + B_{\alpha}^{-1}S_{\alpha},\quad \tilde{\mathbf{b}}^{(m)}= B_{\alpha}^{-1}\mathbf{b}^{(m)}.
\end{equation*}
Although replacing the coefficient matrices $A_{\alpha}^{(m)}$ with $\tilde{A}_{\alpha}^{(m)}$ does not increase the computational cost (both require $\mathcal{O}(dN^d\log N)$ operations), the structure of the resulting linear system \eqref{equi} is more conducive to analyzing the convergence rate of the preconditioned GMRES method, which motivates us to solve Eq. \eqref{equi} instead.

To improve the computational efficiency, we define the $\tau$-preconditioner $P_{\alpha}$ for matrix $\tilde{A}_{\alpha}^{(m)}$ in \eqref{equi} as follows.
\begin{equation}
\label{Pa}
P_{\alpha}=\bar{e}I_{N^d} + B_{\alpha}^{-1/2}\tau(S_{\alpha})B_{\alpha}^{-1/2},
\end{equation}
where
\begin{equation}\label{tauSalpha}
\tau(S_{\alpha}) = \eta_1(B_{\alpha_d} \otimes \cdots \otimes B_{\alpha_{2}} \otimes \tau(S_{\alpha_1})) + \cdots + \eta_d(\tau(S_{\alpha_d}) \otimes B_{\alpha_{d-1}} \otimes \cdots \otimes B_{\alpha_1})
\end{equation}
and $I_{N^d}$ is the identity matrix of order $N^d \times N^d$, and
 \begin{equation}
 \label{bard}
\bar{e}=\frac{\hat{e}+\check{e}}{2}
 \end{equation}
with
\begin{equation}
 \label{dhat} \hat{e}=\max\limits_{\mathbb{G},m}e(\mathbb{G},t_{m+1/2}),~~\check{e}=\min\limits_{\mathbb{G},m}e(\mathbb{G},t_{m+1/2}).
 \end{equation} 

 An interesting property is that $\tau$-matrices form a commutative matrix-algebra and hence
\begin{equation*}
 B_{\alpha}^{-1/2}\tau(S_{\alpha})B_{\alpha}^{-1/2}= B_{\alpha}^{-1}\tau(S_{\alpha})= \tau(S_{\alpha}) B_{\alpha}^{-1}.
\end{equation*}

This implies that computing the matrix-vector multiplications $B_{\alpha}^{-1/2}\tau(S_{\alpha})B_{\alpha}^{-1/2}\mathbf{r}$ and
$B_{\alpha}^{-1}\tau(S_{\alpha})\mathbf{r}$ demands the same operations.

Before proceeding with the analysis, we first introduce several relevant lemmas.

 \begin{lemma}[\cite{Laub1}]
\label{le:eig}
Assume $A\in \mathbb{R}^{n\times n}$ has eigenvalues $\{\lambda_i\}_{i=1}^n$ and $B\in \mathbb{R}^{m\times m}$ has eigenvalues $\{\mu_j\}_{j=1}^m$. Then, the $mn$ eigenvalues of $A\otimes B$ are
\[\lambda_1\mu_1,\ldots,\lambda_1\mu_m,\lambda_2\mu_1,\ldots,\lambda_2\mu_m,\ldots,\lambda_n\mu_1,\ldots,\lambda_n\mu_m.\]
\end{lemma}
\begin{lemma}[{\cite{HuangX2022}}]
\label{tauT}
 For a symmetric Toeplitz matrix, $T\in\mathbb{R}^{n\times n}$, with $[t_0,t_1,\cdots,t_{n-1}]^\top\in\mathbb{R}^{n}$ being the first column, $\tau(T)$ and $H(T)$ are given as in \eqref{tau}. If
 $$
 t_{0}>0,~~t_{k}< 0,~~ k=1,2,\cdots, n-1, ~~\mbox{and} ~~
t_0+2\sum_{k=1}^{n-1} t_{k} >0,
$$
then $\tau(T)$ is symmetric positive definite, and
$$
\left\|\tau(T)^{-1/2}H(T)\tau(T)^{-1/2}\right\|_2<\frac12,
$$
which implies that
$$
-\frac12\mathbf{v}^{*}\tau(T)\mathbf{v}<\mathbf{v}^{*}H(T)\mathbf{v}<\frac12\mathbf{v}^{*}\tau(T)\mathbf{v},
~~\mbox{for~~any}~~\mathbf{v}\in \mathbb{C}^n, \mathbf{v}\neq 0.
$$
\end{lemma}

The next lemma is about the positive definiteness of $P_{\alpha}$.

\begin{lemma}
\label{PSPD}
The preconditioner $P_{\alpha}$ defined in \eqref{Pa} is symmetric positive definite.
\end{lemma}
\begin{proof}
From (\ref{As1}) and Lemma \ref{tauT}, it follows that $\tau(S_{\alpha_i})$ is symmetric positive definite. By Lemma \ref{le:eig}, all the eigenvalues of matrix $\tau(S_{\alpha})$ are positive, which means that $\tau(S_{\alpha})$ is symmetric positive definite. Therefore, since congruence preserves positive definiteness, $B_{\alpha}^{-1/2}\tau(S_{\alpha})B_{\alpha}^{-1/2}$ is positive definite, and hence $P_{\alpha}$ is symmetric positive definite.
\end{proof}

By Lemma \ref{PSPD}, $P_{\alpha}^{-1}$ exists; hence so does $P_{\alpha}^{-1/2}$.  Consequently, we can solve the following two-sided preconditioned linear systems.
\begin{equation}
\label{Pls}
P_{\alpha}^{-1/2}\tilde{A}_{\alpha}^{(m)}P_{\alpha}^{-1/2}\tilde{\mathbf{u}}^{(m+1)}
=P_{\alpha}^{-1/2}\tilde{\mathbf{b}}^{(m)},
\end{equation}
where
$$
\tilde{\mathbf{u}}^{(m+1)}=P_{\alpha}^{1/2}\mathbf{u}^{(m+1)}, \quad 0\leq m \leq M-1.
$$
Evidently, $P_{\alpha}$ is a $\tau$-matrix. According to the matrix diagonalization \eqref{diag}, we can infer that each matrix-vector multiplication related to $P_{\alpha}^{-1/2}\tilde{A}_{\alpha}^{(m)}P_{\alpha}^{-1/2}$ necessitates $\mathcal{O}(dN^d\log N)$ real operations. Moreover, once $\tilde{\mathbf{u}}^{(m + 1)}$ is provided, the computation of $\mathbf{u}^{(m+1)}$ also requires $\mathcal{O}(dN^d\log N)$ real operations.

Unlike the conjugate gradient (CG) method, the GMRES method may still require a large number of iterations even if every eigenvalues of the coefficient matrix are equal to 1. This is illustrated by applying the GMRES method to a linear system whose coefficient matrix is tridiagonal matrix, namely $\mathrm{tridiag}(-1,1,0)$. Analyzing the convergence rate of GMRES method solely through the distribution of eigenvalues is insufficient.

To analyze the convergence of the PGMRES method, we need some notations and known results. Denote the numerical range of a matrix $A\in \mathbb{C}^{n\times n}$ by
$$
\mathcal{W}(A)=\left\{\frac{\mathbf{v}^{*}A\mathbf{v}}{\mathbf{v}^{*}\mathbf{v}}\mid
\mathbf{v}\in \mathbb{C}^n,\mathbf{v}\neq 0\right\},
$$
where ``$*$" means the conjugate transpose. Let $\theta\in (0,\pi/2)$ be defined by
\begin{equation}
\label{numrag}
\cos(\theta)=\frac{v(A)}{r(A)},\quad v(A)=\min\limits_{\rho\in \mathcal{W}(A)}|\rho|,\quad
r(A)=\max\limits_{\rho\in \mathcal{W}(A)}|\rho|.
\end{equation}
The following convergence results hold for the GMRES method.
\begin{lemma}[{\cite{LT1,QPS1}}]
\label{bound}
If $A\in \mathbb{C}^{n\times n}$ is such that $0\notin \mathcal{W}(A)$, and $\theta\in (0,\pi/2)$ is given as in \eqref{numrag}, then for $k\geq 1$, we find that
\begin{equation}
\frac{\|\mathbf{r}_k\|_2}{\|\mathbf{r}_0\|_2}\leq (2+\rho_{\theta}){\rho_{\theta}^k},\quad \rho_{\theta}=2\sin\left(\frac{\theta}{4-2\theta/\pi}\right)<\sin\theta,
\end{equation}
where $\mathbf{r}_k$ is the $k$th residual of the GMRES method for solving the linear system $A\mathbf{u}=\mathbf{b}$.
\end{lemma}

\begin{lemma}
\label{Heig}
Let $B_{\alpha_i}$ be defined by \eqref{Ha}, then
$$
\mathrm{cond}(B_{\alpha_i}^{1/2})=\left\|B_{\alpha_i}^{-1/2}\right\|_2\left\|B_{\alpha_i}^{1/2}\right\|_2 \leq \frac{\sqrt{6}}{2},
$$
for $1 \leq i \leq d$.
\end{lemma}
\begin{proof}
From \cite{Leveque1}, we have that the eigenvalues of the matrix $B=\mathrm{tridiag}(1,-2,1) \in \mathbb{R}^{N\times N}$ in decreasing order are

$$\lambda_{k}(B)=-4\sin^2\left(\frac{k\pi}{2N+2}\right),\quad k=1,2,\cdots,N.$$

Obviously,
$$
B_{\alpha_i}=I_N+\frac{\alpha_i}{24}B,\quad 1 \leq i \leq d.
$$
Then, the eigenvalues of $B_{\alpha_i}$ are given by
$$\frac23 <1-\frac{\alpha_i}{6}< \lambda_{k}(B_{\alpha_i})=1+\frac{\alpha_i}{24}\lambda_{k}(B)< 1.$$
It implies that  $\frac{\sqrt{6}}{3}<\lambda_{k}(B_{\alpha_i}^{1/2})<1$, which completes the proof.
\end{proof}

\begin{lemma}
\label{rate}
For the positive numbers ~$a_i$,~$b_i$,~$1\leq\ i\leq m$, we have
\begin{eqnarray*}
    \min\limits_{1\leq i\leq m} {\dfrac{a_i}{b_i}}  \leq   \left(\sum\limits_{i=1}^m a_i\right) /\left(\sum\limits_{i=1}^m b_i\right)    \leq  \max\limits_{1\leq i\leq m} {\dfrac{a_i}{b_i}}.
\end{eqnarray*}
\end{lemma}
In the following, we utilize Theorem \ref{GMRES} to show the convergence rate of the GMRES method for the preconditioned linear systems \eqref{Pls}.

\begin{theorem}
\label{GMRES}
For any given $m\in \mathbb{N}$, $0\leq m\leq M$, the GMRES method for the preconditioned linear systems \eqref{Pls} has a convergence rate independent of grid size $N^d$, i.e., the residuals generated by GMRES solver satisfy
$$
\frac{\|\mathbf{r}_k\|_{2}}{\|\mathbf{r}_0\|_2}\leq (2+c){c^k},
$$
where
\begin{align}
c=&\max\left\{\sqrt{1-\frac{\check{e}^2}{\hat{e}^2}},~\sqrt{\frac{16\sqrt{6}}{(4+\sqrt{6})^2}},~
\sqrt{1-\frac{(\hat{e}+\check{e})^2(11-4\sqrt{6})}{32\hat{e}^2}},~
\sqrt{1-\frac{32\check{e}^2}{(11+4\sqrt{6})(\hat{e}+\check{e})^2}}
\right\}, \notag
\end{align} $\mathbf{r}_k=P_{\alpha}^{-1/2}\tilde{\mathbf{b}}^{(m)}-P_{\alpha}^{-1/2}\tilde{A}_{\alpha}^{(m)}P_{\alpha}^{-1/2}\tilde{\mathbf{u}}^{(k)}_{*}$, $\tilde{\mathbf{u}}^{(k)}_{*}$ denotes the $k$-th iteration by GMRES method, $\tilde{\mathbf{u}}^{(0)}_{*}\in \mathbb{R}^{N^d}$ denotes an arbitrary initial guess, and $\hat{e}$ and $\check{e}$ are defined by \eqref{dhat}.

\end{theorem}
\begin{proof}
Denote $\breve{A}_{\alpha}^{(m)}=P_{\alpha}^{-1/2}\tilde{A}_{\alpha}^{(m)}P_{\alpha}^{-1/2}$.
According to \eqref{numrag}, we need to investigate the numerical radius $r(\breve{A}_{\alpha})$ and the distance of $\mathcal{W}(\breve{A}_{\alpha})$ from the origin. For any $\mathbf{y}\in \mathbb{C}^{N^d}$ and $\mathbf{y}\neq 0$, we deduce the following explicit expression
\begin{equation}
\label{Theq1}
\begin{aligned}
\frac{\mathbf{y}^{*}\breve{A}_{\alpha}^{(m)}\mathbf{y}}{\mathbf{y}^{*}\mathbf{y}}
&=1+\frac{\mathbf{y}^{*}P_{\alpha}^{-1/2}\left(E^{(m+1/2)}-\bar{e}I_{N^d} + B_{\alpha}^{-1}H(S_{\alpha})\right)P_{\alpha}^{-1/2}\mathbf{y}}{\mathbf{y}^{*}\mathbf{y}}\\
&=1+\frac{\mathbf{v}^{*}\left(E^{(m+1/2)}-\bar{e}I_{N^d} \right)\mathbf{v}+\mathbf{v}^{*}B_{\alpha}^{-1}H(S_{\alpha})\mathbf{v}}
{\mathbf{v}^{*}P_{\alpha}\mathbf{v}}\\
&=1+\frac{\mathbf{v}^{*}\left(E^{(m+1/2)}-\bar{e}I_{N^d} \right)\mathbf{v}+\mathbf{v}^{*}B_{\alpha}^{-1}H(S_{\alpha})\mathbf{v}}
{\bar{e}\mathbf{v}^{*}\mathbf{v} + \mathbf{v}^{*} B_{\alpha}^{-1}\tau(S_{\alpha})\mathbf{v}}\\
&=1+\frac{\mathbf{v}^{*}\left(E^{(m+1/2)}-\bar{e}I_{N^d} \right)\mathbf{v}+\eta_1\mathbf{v}^{*}(I_{N} \otimes \cdots \otimes I_{N} \otimes B_{\alpha_1}^{-1} H(S_{\alpha_1}))\mathbf{v}}
{\bar{e}\mathbf{v}^{*}\mathbf{v} + \eta_1\mathbf{v}^{*}(I_{N} \otimes \cdots \otimes I_{N} \otimes B_{\alpha_1}^{-1} \tau(S_{\alpha_1}))\mathbf{v}}\\
&\ \frac{+ \cdots + \eta_d \mathbf{v}^{*}( B_{\alpha_d}^{-1} H(S_{\alpha_d}) \otimes \cdots \otimes I_{N} \otimes I_{N})\mathbf{v}}
{+ \cdots + \eta_d \mathbf{v}^{*}( B_{\alpha_d}^{-1} \tau(S_{\alpha_d}) \otimes \cdots \otimes I_{N} \otimes I_{N})\mathbf{v}},
\end{aligned}
\end{equation}
where $\mathbf{v}=P_{\alpha}^{-1/2}\mathbf{y}$. According to Lemmas \ref{PSPD}, it follows that
\begin{equation}
\label{adpostive}
\mathbf{v}^{*}P_{\alpha}\mathbf{v}>0,~\text{and}~\mathbf{v}^{*} B_{\alpha}^{-1}\tau(S_{\alpha})\mathbf{v}>0.
\end{equation}
By Lemma \ref{rate} and inequations \eqref{adpostive}, it holds 
\begin{equation}
\begin{aligned}
&\left|\frac{\mathbf{v}^{*}\left(E^{(m+1/2)}-\bar{e}I_{N^d} \right)\mathbf{v}+\mathbf{v}^{*}B_{\alpha}^{-1}H(S_{\alpha})\mathbf{v}}
{\bar{e}\mathbf{v}^{*}\mathbf{v} + \mathbf{v}^{*} B_{\alpha}^{-1}\tau(S_{\alpha})\mathbf{v}}\right|\\
\leq &
\frac{\left|\mathbf{v}^{*}\left(E^{(m+1/2)}-\bar{e}I_{N^d} \right)\mathbf{v}\right|+\left|\mathbf{v}^{*}B_{\alpha}^{-1}H(S_{\alpha})\mathbf{v}\right|}
{\bar{e}\mathbf{v}^{*}\mathbf{v} + \mathbf{v}^{*} B_{\alpha}^{-1}\tau(S_{\alpha})\mathbf{v}}\\
%
%
\leq &
\max\left\{\frac{\left|\mathbf{v}^{*}\left(E^{(m+1/2)}-\bar{e}I_{N^d} \right)\mathbf{v}\right|}{\bar{e}\mathbf{v}^{*}\mathbf{v}},
\frac{\left|\mathbf{v}^{*}B_{\alpha}^{-1}H(S_{\alpha})\mathbf{v}\right|}{\mathbf{v}^{*} B_{\alpha}^{-1}\tau(S_{\alpha})\mathbf{v}}\right\}\\
\leq&
\max\left\{\frac{\left|\mathbf{v}^{*}\left(E^{(m+1/2)}-\bar{e}I_{N^d} \right)\mathbf{v}\right|}{\bar{e}\mathbf{v}^{*}\mathbf{v}},
\frac{\left|\mathbf{v}^{*}(I_{N} \otimes \cdots \otimes I_{N} \otimes B_{\alpha_1}^{-1} H(S_{\alpha_1}))\mathbf{v}\right|}{\mathbf{v}^{*}(I_{N} \otimes \cdots \otimes I_{N} \otimes B_{\alpha_1}^{-1} \tau(S_{\alpha_1}))\mathbf{v}},\right.\\
&\frac{\left|\mathbf{v}^{*}(I_{N} \otimes \cdots \otimes I_{N} \otimes B_{\alpha_2}^{-1} H(S_{\alpha_2})\otimes I_{N})\mathbf{v}\right|}{\mathbf{v}^{*}(I_{N} \otimes \cdots \otimes I_{N} \otimes B_{\alpha_2}^{-1} \tau(S_{\alpha_2})\otimes I_{N}) \mathbf{v}},\cdots,\left.\frac{\left|\mathbf{v}^{*}( B_{\alpha_d}^{-1} H(S_{\alpha_d}) \otimes \cdots \otimes I_{N} \otimes I_{N})\mathbf{v}\right|}{\mathbf{v}^{*}( B_{\alpha_d}^{-1} \tau(S_{\alpha_d}) \otimes \cdots \otimes I_{N} \otimes I_{N})\mathbf{v}}\right\}.\\
\end{aligned}
\end{equation}
Obviously, 
\begin{equation}
\label{Theq2}
\frac{\left|\mathbf{v}^{*}\left(E^{(m+1/2)}-\bar{e}I_{N^d}\right)\mathbf{v}\right|}{\bar{e}\mathbf{v}^{*}\mathbf{v}}
\leq \frac{\hat{e}-\check{e}}{\hat{e}+\check{e}}<1,
\end{equation}
where $\hat{e}$ and $\check{e}$ are defined by \eqref{dhat}. Next, we analyze the upper bound of the term $\frac{\left| \mathbf{v}^{*}(I_{N} \otimes \cdots \otimes I_{N} \otimes B_{\alpha_1}^{-1} H(S_{\alpha_1}))\mathbf{v} \right|}{\mathbf{v}^{*}(I_{N} \otimes \cdots \otimes I_{N} \otimes B_{\alpha_1}^{-1} \tau(S_{\alpha_1}))\mathbf{v}}$. The upper bounds of the remaining terms coincide with this one, and their derivations are analogous; hence we shall omit them in the following. By \cite[{Theorem 1.2}]{GT1} and the commutativity of $\tau$ matrices,  we get
\begin{equation}
\label{addbound}
\begin{aligned}
&\frac{\left| \mathbf{v}^{*}(I_{N} \otimes \cdots \otimes I_{N} \otimes B_{\alpha_1}^{-1} H(S_{\alpha_1}))\mathbf{v} \right|}{\mathbf{v}^{*}(I_{N} \otimes \cdots \otimes I_{N} \otimes B_{\alpha_1}^{-1} \tau(S_{\alpha_1}))\mathbf{v}}\\
=&\frac{\left| \mathbf{v}^{*}(I_{N} \otimes \cdots \otimes I_{N} \otimes B_{\alpha_1}^{-1} H(S_{\alpha_1}))\mathbf{v} \right|}{\mathbf{v}^{*}(I_{N} \otimes \cdots \otimes I_{N} \otimes B_{\alpha_1}^{-1/2} \tau(S_{\alpha_1})B_{\alpha_1}^{-1/2})\mathbf{v}}\\
=& \frac{\left|\mathbf{z}^{*} (I_{N} \otimes \cdots \otimes \tau(S_{\alpha_1})^{-1/2} B_{\alpha_1}^{-1/2} H(S_{\alpha_1}) B_{\alpha_1}^{1/2} \tau(S_{\alpha_1})^{-1/2} )\mathbf{z}\right|}{\mathbf{z}^{*}\mathbf{z}} \\
\leq& \left\|\tau(S_{\alpha_1})^{-1/2}B_{\alpha_1}^{-1/2}H(S_{\alpha_1})B_{\alpha_1}^{1/2}\tau(S_{\alpha_1})^{-1/2}\right\|_2\\
=&\left\|B_{\alpha_1}^{-1/2}\tau(S_{\alpha_1})^{-1/2}H(S_{\alpha_1})\tau(S_{\alpha_1})^{-1/2}B_{\alpha_1}^{1/2}\right\|_2\\
\leq& \mathrm{cond}(B_{\alpha_1}^{1/2})\left\|\tau(S_{\alpha_1})^{-1/2}H(S_{\alpha_1})\tau(S_{\alpha_1})^{-1/2}\right\|_2,
\end{aligned}
\end{equation}
where $$\mathbf{z} = (I_{N} \otimes \underbrace{\cdots}_{d-3} \otimes I_{N} \otimes  \tau(S_{\alpha_1})^{1/2}B_{\alpha_1}^{-1/2})\mathbf{v}.$$ 

 According to Lemma \ref{tauT}, Lemma \ref{Heig} and the estimate \eqref{addbound}, we have
\begin{equation*}
\frac{\left| \mathbf{v}^{*}(I_{N} \otimes \cdots \otimes I_{N} \otimes B_{\alpha_1}^{-1} H(S_{\alpha_1}))\mathbf{v} \right|}{\mathbf{v}^{*}(I_{N} \otimes \cdots \otimes I_{N} \otimes B_{\alpha_1}^{-1} \tau(S_{\alpha_1}))\mathbf{v}}
\leq \frac{\sqrt{6}}{2}\times\frac{1}{2}=\frac{\sqrt{6}}{4}<1.
\end{equation*}
Therefore, it follows that 
\begin{equation}
\label{Theq3}
\left|\frac{\mathbf{v}^{*}\left(E^{(m+1/2)}-\bar{e}I_{N^d} \right)\mathbf{v}+\mathbf{v}^{*}B_{\alpha}^{-1}H(S_{\alpha})\mathbf{v}}
{\bar{e}\mathbf{v}^{*}\mathbf{v} + \mathbf{v}^{*} B_{\alpha}^{-1}\tau(S_{\alpha})\mathbf{v}}\right|
\leq
\max\left\{ \frac{\hat{e}-\check{e}}{\hat{e}+\check{e}},
\frac{\sqrt{6}}{4}
\right\}.
\end{equation}
From \eqref{Theq3}, we have
\begin{equation}
\label{Theq4}
\begin{aligned}
\min\left\{\frac{2\check{e}}{\hat{e}+\check{e}},\frac{4-\sqrt{6}}{4}\right\}
& \leq 1-\max_{\mathbf{v} \in \mathbb{C}^{N^d}} \left\{\frac{\left|\mathbf{v}^{*}\left(E^{(m+1/2)}-\bar{e}I_{N^d}\right)\mathbf{v}\right|}
{\bar{e}\mathbf{v}^{*}\mathbf{v} },\frac{\left|\mathbf{v}^{*}B_{\alpha}^{-1}H(S_{\alpha})\mathbf{v}\right|}{\mathbf{v}^{*} B_{\alpha}^{-1}\tau(S_{\alpha})\mathbf{v}}\right\}\\
&\leq
1-\frac{\left|\mathbf{v}^{*}\left(E^{(m+1/2)}-\bar{e}I_{N^d} \right)\mathbf{v}\right|+\left|\mathbf{v}^{*}B_{\alpha}^{-1}H(S_{\alpha})\mathbf{v}\right|}
{\bar{e}\mathbf{v}^{*}\mathbf{v} + \mathbf{v}^{*} B_{\alpha}^{-1}\tau(S_{\alpha})\mathbf{v}}\\
&\leq v(\breve{A}_{\alpha})
\leq \frac{\left|\mathbf{y}^{*}\breve{A}_{\alpha}^{(m)}\mathbf{y}\right|}{\mathbf{y}^{*}\mathbf{y}}
\leq r(\breve{A}_{\alpha})\\
& \leq 1+\frac{\left|\mathbf{v}^{*}\left(E^{(m+1/2)}-\bar{e}I_{N^d}\right)\mathbf{v}\right|+\left|\mathbf{v}^{*}B_{\alpha}^{-1}H(S_{\alpha})\mathbf{v}\right|}
{\bar{e}\mathbf{v}^{*}\mathbf{v} + \mathbf{v}^{*} B_{\alpha}^{-1}\tau(S_{\alpha})\mathbf{v}}\\
& \leq 1+\max_{\mathbf{v} \in \mathbb{C}^{N^d}} \left\{\frac{\left|\mathbf{v}^{*}\left(E^{(m+1/2)}-\bar{e}I_{N^d}\right)\mathbf{v}\right|}
{\bar{e}\mathbf{v}^{*}\mathbf{v} },\frac{\left|\mathbf{v}^{*}B_{\alpha}^{-1}H(S_{\alpha})\mathbf{v}\right|}{\mathbf{v}^{*} B_{\alpha}^{-1}\tau(S_{\alpha})\mathbf{v}}\right\}\\
&\leq \max\left\{\frac{2\hat{e}}{\hat{e}+\check{e}},\frac{4+\sqrt{6}}{4}\right\}.\\
\end{aligned}
\end{equation}

Finally, by Lemma \ref{bound} and \eqref{Theq4}, we get
\begin{align}
c=&\max\left\{\sqrt{1-\frac{\check{e}^2}{\hat{e}^2}},~\sqrt{\frac{16\sqrt{6}}{(4+\sqrt{6})^2}},~
\sqrt{1-\frac{(\hat{e}+\check{e})^2(11-4\sqrt{6})}{32\hat{e}^2}},~
\sqrt{1-\frac{32\check{e}^2}{(11+4\sqrt{6})(\hat{e}+\check{e})^2}}
\right\}. \notag
\end{align}
\end{proof}
We supplement the properties of the proposed two-sided preconditioned solver with some remarks.

\begin{remark}
The proposed preconditioner $P_{\alpha}$ is independent of the time step. In the whole calculation, the eigenvalues of $P_{\alpha}$ need only to be evaluated once.
\end{remark}

\begin{remark}
If the coefficient matrix of other linear systems can be represented in a tensor product form analogous to \eqref{tls} (not necessarily arising from the discretization of fractional-order PDEs), characterized by a specific (block) tridiagonal matrix and a strictly diagonally dominant space-discretization matrix, we can infer that the PGMRES method all adhere to the same bound, as stated in Theorem \ref{GMRES}.
\end{remark}

In the previous text, we established the convergence rate of the two-sided preconditioned GMRES method. However, each iteration of the two-sided preconditioned GMRES method is computationally more expensive than its one-sided counterpart. To further reduce the cost, we now examine the efficiency of the one-sided preconditioned solver, beginning with its definition:
\begin{equation}
\label{one-sided}
P_{\alpha}^{-1}\tilde{A}_{\alpha}^{(m)}\mathbf{u}^{(m+1)}
=P_{\alpha}^{-1}\tilde{\mathbf{b}}^{(m)},~~0\leq m\leq M-1.
\end{equation}

Next, we demonstrate the residual relationship between the one-sided preconditioned solver \eqref{one-sided} and the two-sided preconditioned solver \eqref{Pls} to better illustrate that the numerical range of the one-sided preconditioner is also controlled by this relationship. 

Recall from \cite{Linx2023}, for a square matrix $E \in \mathbb{R}^{n \times n}$ and a vector $\mathbf{v} \in \mathbb{R}^{n \times 1}$, a Krylov subspace of degree $j \geq 1$ is defined as follows,
$$
\mathcal{K}_j(E, \mathbf{v}):=\operatorname{span}\left\{\mathbf{v}, E\mathbf{v}, E^2 \mathbf{v}, \ldots, E^{j-1} \mathbf{v}\right\}.
$$
For a set $\mathcal{S}$ and a point $\mathbf{z}$, we denote
$$
\mathbf{z}+\mathcal{S}:=\{\mathbf{z}+\mathbf{v} \mid \mathbf{v} \in \mathcal{S}\} .
$$

\begin{lemma} [{\cite{Saad86}}]\label{Krylov}
 For a non-singular $n \times n$ real linear system $A \mathbf{y}=\mathbf{b}$, let $\mathbf{y}^{(j)}$ be the iterative solution by GMRES at $j$-th $(j \geq 1)$ iteration step with $\mathbf{y}^{(0)}$ as initial guess. Then, the $j$-th iteration solution $\mathbf{y}^{(j)}$ minimizes the residual error over the Krylov subspace $\mathcal{K}_j\left(A, \mathbf{r}_0\right)$ with $\mathbf{r}_0=\mathbf{b}-A \mathbf{y}^{(0)}$, i.e.,
$$
\mathbf{y}^{(j)}=\underset{\mathbf{q} \in \mathbf{y}^{(0)}+\mathcal{K}_j\left(A, \mathbf{r}_0\right)}{\arg \min }\|\mathbf{b}-A\mathbf{q}\|_2 .
$$
\end{lemma}

\begin{theorem}\label{one-sidethm}
Let $\tilde{\mathbf{u}}^{(0)}$ be the initial guess for \eqref{Pls}. Let $\mathbf{u}^{(0)}:=P_{\alpha}^{-1/2} \tilde{\mathbf{u}}^{(0)}$ be the initial guess for \eqref{one-sided}. Let $\mathbf{u}^{(j)} (\tilde{\mathbf{u}}^{(j)}$, respectively$)$ be the $j$-th $(j \geq 1)$ iteration solution derived by applying GMRES solver to \eqref{one-sided}$($\eqref{Pls}, respectively$)$ with $\mathbf{u}^{(0)}(\tilde{\mathbf{u}}^{(0)}$, respectively$)$ as initial guess. Then,
$$
\left\|\mathbf{r}_{j}\right\|_2 \leq \frac{1}{\sqrt{\bar{e}}}\left\|\tilde{\mathbf{r}}_{j}\right\|_2,
$$
where $\mathbf{r}_{j}:=P_\alpha^{-1} \tilde{\mathbf{b}}^{(m)} - P_\alpha^{-1} \tilde{A}_\alpha^{(m)} \mathbf{u}^{(j)} (\tilde{\mathbf{r}}_{j}:=P_\alpha^{-1/2} \tilde{\mathbf{b}}^{(m)}-P_\alpha^{-1/2} \tilde{A}_\alpha^{(m)} P_\alpha^{-1/2} \tilde{\mathbf{u}}^{(j)}$, respectively$)$ denotes the residual vector at $j$-th GMRES iteration for \eqref{one-sided} $($\eqref{Pls}, respectively$)$; $\bar{e}>0$ defined in \eqref{bard} is a constant independent of grid size $N$ and $M$.
\end{theorem}

\begin{proof}
Denote $\breve{A}_{\alpha}^{(m)}=P_{\alpha}^{-1/2}\tilde{A}_{\alpha}^{(m)}P_{\alpha}^{-1/2}$. By applying Lemma \ref{Krylov} to \eqref{Pls}, we infer that

\begin{equation}\label{eq3.17} 
   \tilde{\mathbf{u}}^{(j)}-\tilde{\mathbf{u}}^{(0)} \in \mathcal{K}_{j}\left(\breve{A}_{\alpha}^{(m)}, \tilde{\mathbf{r}}_0\right), 
\end{equation}
where $\tilde{\mathbf{r}}_0=P_\alpha^{-1/2} \tilde{\mathbf{b}}^{(m)} - \breve{A}_{\alpha}^{(m)} \tilde{\mathbf{u}}^{(0)}$. Notice that $\left(\breve{A}_{\alpha}^{(m)} \right)^k=P_\alpha^{1/2}\left(P_\alpha^{-1} \tilde{A}_\alpha^{(m)} \right)^k P_\alpha^{-1/2}$ for each $k \geq 0$. Therefore,

\begin{equation}\label{eq3.18}
\begin{aligned}
\mathcal{K}_{j}\left(\breve{A}_{\alpha}^{(m)}, \tilde{\mathbf{r}}_0\right) & =\operatorname{span}\left\{\left(\breve{A}_{\alpha}^{(m)}\right)^k\left(P_\alpha^{-1/2} \tilde{\mathbf{b}}^{(m)} - \breve{A}_{\alpha}^{(m)}\tilde{\mathbf{u}}^{(0)}\right)\right\}_{k=0}^{j-1} \\
& =\operatorname{span}\left\{P_\alpha^{1/2}\left(P_\alpha^{-1} \tilde{A}_\alpha^{(m)}\right)^k P_\alpha^{-1/2}\left(P_\alpha^{-1/2} \tilde{\mathbf{b}}^{(m)}-P_\alpha^{-1/2} \tilde{A}_\alpha^{(m)} P_\alpha^{-1/2} \tilde{\mathbf{u}}^{(0)}\right)\right\}_{k=0}^{j-1} \\
& =\operatorname{span}\left\{P_\alpha^{1/2}\left(P_\alpha^{-1} \tilde{A}_\alpha^{(m)}\right)^k\left(P_\alpha^{-1} \tilde{\mathbf{b}}^{(m)}-P_\alpha^{-1} \tilde{A}_\alpha^{(m)} \mathbf{u}^{(0)}\right)\right\}_{k=0}^{j-1}. 
\end{aligned}
\end{equation}
Furthermore, we easily obtain

\begin{equation}\label{eq3.19}
\begin{aligned}
P_\alpha^{-1/2} \tilde{\mathbf{u}}^{(j)}-\mathbf{u}^{(0)}&=P_\alpha^{-1/2}\left(\tilde{\mathbf{u}}^{(j)}-\tilde{\mathbf{u}}^{(0)}\right) \\
&\in \operatorname{span}\left\{\left(P_\alpha^{-1} \tilde{A}_\alpha^{(m)}\right)^k\left(P_\alpha^{-1} \tilde{\mathbf{b}}^{(m)}-P_\alpha^{-1} \tilde{A}^{(m)}_\alpha \mathbf{u}^{(0)}\right)\right\}_{k=0}^{j-1} \\
&=\mathcal{K}_{j}\left(P_\alpha^{-1} \tilde{A}_\alpha^{(m)}, \mathbf{r}_0\right),
\end{aligned}
\end{equation}
where $\mathbf{r}_0=P_\alpha^{-1} \tilde{\mathbf{b}}^{(m)}-P_\alpha^{-1} \tilde{A}_\alpha^{(m)} \mathbf{u}^{(0)}$. In other words,
$$
P_\alpha^{-1/2}  \tilde{\mathbf{u}}^{(j)} \in \mathbf{u}^{(0)}+\mathcal{K}_{j}\left(P_\alpha^{-1} \tilde{A}_\alpha^{(m)}, \mathbf{r}_0\right).
$$
By applying Lemma \ref{Krylov} to \eqref{one-sided}, we know that

\begin{equation*}
    \mathbf{u}^{(j)}=\underset{\mathbf{q} \in \mathbf{u}^{(0)} + \mathcal{K}_{j}\left( P_\alpha^{-1} \tilde{A}_\alpha^{(m)}, \mathbf{r}_0\right)}{\arg \min }\left\|P_\alpha^{-1} \tilde{\mathbf{b}}^{(m)}-P_\alpha^{-1} \tilde{A}_\alpha^{(m)} \mathbf{q}\right\|_2 .
\end{equation*}
It is easy to check that $\bar{e} I_{N^d} + B_{\alpha}^{-1/2} \tau(S_\alpha) B_{\alpha}^{-1/2} > \bar{e} I_{N^d}$ from Lemma \ref{PSPD}. Then we get

\begin{equation*}
\begin{aligned}
\left\| \left(\bar{e} I_{N^d} + B_{\alpha}^{-1/2} \tau(S_\alpha) B_{\alpha}^{-1/2} \right)^{-1} \right\|_2 = \left\| P_\alpha^{-1}\right\|_2 \leq \frac{1}{\bar{e}}.
\end{aligned}
\end{equation*}
Therefore,

\begin{equation}
\begin{split}
\label{eq3.20}
\left\|\mathbf{r}_{j}\right\|_2=\left\|P_\alpha^{-1} \tilde{\mathbf{b}}^{(m)}-P_\alpha^{-1} \tilde{A}_\alpha^{(m)} \mathbf{u}^{(j)}\right\|_2 & \leq\left\|P_\alpha^{-1} \tilde{\mathbf{b}}^{(m)} - P_\alpha^{-1} \tilde{A}_\alpha^{(m)}  P_\alpha^{-1/2} \tilde{\mathbf{u}}^{(j)}\right\|_2 \\
& \leq \left\|P_\alpha^{-1/2} \right\|_2 \left\|\tilde{\mathbf{r}}_{j}\right\|_2\\
&=\frac{1}{\sqrt{\bar{e}}}\left\|\tilde{\mathbf{r}}_{j}\right\|_2.
\end{split}
\end{equation}

\end{proof}

We conclude this section with the following remark.
\begin{remark}
Although the research in \cite{Linx2023} centers on the comparable convergence analysis of the one- and two-sided PGMRES method, the coefficient $\check{c}$ in that study tends to 0. Consequently, the preconditioned solver employed may exhibit suboptimal performance when dealing with matrices of large size.
In stark contrast, the coefficient $\bar{e}$ in \eqref{eq3.20} meets the condition \eqref{bard}, remaining independent of the spatial grid. As long as $\bar{e} \geq 1$, we can guarantee that $\left\|\mathbf{r}_{j}\right\|_2 \leq \left\|\tilde{\mathbf{r}}_{j}\right\|_2$. 

\end{remark}

\section{Numerical results}
\label{Numer}
In this section, numerical examples are presented to illustrate the efficiency of the proposed methods. All numerical experiments were conducted using MATLAB R2023a on 11th Gen Intel(R) Core(TM) i7-11700K @ 3.60GHz and 48.0 GB of RAM.

To show the effectiveness of our proposed PGMRES solvers, include one- and two-sided preconditioned versions. In any case, the GMRES solver is implemented with the built-in function of MATLAB. Specifically, the stopping criterion for GMRES is set as $\left\|\mathbf{r}_k\right\|_2 \leq \text{tol} \left\|\mathbf{r}_0\right\|_2$, where $\mathbf{r}_k$ denotes the residual vector in the $k$-th GMRES iteration for $k \geq 1$ and $\mathbf{r}_0$ denotes the initial residual vector. In the one- and two-dimensional cases, we set the tolerance to $\text{tol} = 10^{-9}$, while in the three-dimensional case, since higher precision is required, the tolerance is set to $\text{tol} = 10^{-10}$. The maximum number of iterations for GMRES is set at 10, 100, 200 in one, two, and three dimensions, respectively. For clarity, to show a fourth-order convergence rate of the proposed scheme, we set $M = N^2$ in all subsequent numerical experiments. For explanations of the symbols appearing in the tables and figures, see Table \ref{table1}, where the norm $|| \cdot||_{L_2}$ denotes the $L_2$-norm.

 \begin{table}[htbp]
\centering
\caption{Explanation of all symbols used in the subsequent numerical experiments}
\label{table1}
\begin{tabular}{ll}
\toprule
Symbol & Explanation \\
\midrule
SPGM & The One-sided Strang PGMRES method to solve \eqref{tls} \\
OPGM & The One-sided PGMRES method (Self-implemented) to solve \eqref{one-sided} \\
TPGM & The Two-sided PGMRES method (Self-implemented) to solve \eqref{Pls} 
\\
Error & $\text { Error }=\left\|\mathbf{u}^*-\mathbf{u}\right\|_{L_2}$, where $\mathbf{u}$ denotes numerical solution and $\mathbf{u}^*$ denotes exact solution 
\\
CPU &  Total CPU time in seconds for solving the whole system 
\\
Iter & The number of iterations averaged over the time steps \\
$\Re(\lambda)$ & The real part of the eigenvalue  $\lambda$ \\
$\Im(\lambda)$ & The imaginary part of the eigenvalue  $\lambda$ \\
\bottomrule
\end{tabular}
\end{table} 

The Strang circulant preconditioner is defined by 
\begin{equation*}
\label{Pcir}
P_{c}=\bar{e}I_{N^d} + c(B_{\alpha})^{-1/2}c(S_{\alpha})c(B_{\alpha})^{-1/2},
\end{equation*}
where $\bar{e}$ is defined by \eqref{bard},
\begin{equation*}
\label{cirSalpha}
\begin{aligned}
c(S_{\alpha}) &= \eta_1(s(B_{\alpha_d}) \otimes \cdots \otimes s(B_{\alpha_{2}}) \otimes s(S_{\alpha_1})) + \cdots + \eta_d(s(S_{\alpha_d}) \otimes s(B_{\alpha_{d-1}}) \otimes \cdots \otimes s(B_{\alpha_1})),\\
c(B_{\alpha})&=s(B_{\alpha_d})\otimes \cdots \otimes s(B_{\alpha_{2}}) \otimes s(B_{\alpha_1}),
\end{aligned}
\end{equation*}
and $s(T)$ is the Strang circulant preconditioner for any given Toeplitz matrix $T$. For further properties of the Strang circulant preconditioner, see \cite{SSC-SIMAX-99}.

It is worth remarking that, as suggested in Refs. \cite{DMS-JCP-16,She2021}, banded preconditioners for linear systems \eqref{equi} can also be constructed by replacing each $S_{\alpha_i}$ with appropriately chosen banded matrices. However, owing to the tensor-product structure, the storage requirements for the incomplete LU factorization of a high-dimensional banded preconditioner become prohibitive, and CPU time far exceeds that of the method proposed herein.  Consequently, numerical results for the banded preconditioners are omitted.

\begin{example}
In the one-dimensional grid of $(N+1)$, instantaneous solute release point sources are established at the internal grid points. The initial concentration $\psi$ of these point sources satisfies $\psi=\psi(x_{1})=100x_1^4(1-x_1)^4$, and the initial concentration of other boundary points is $0$. We consider a diffusion coefficient $\kappa_1=100$, a variable coefficient $e(x_1,t)=(x_1^2 + e^{-t})/50$ and exact solution $u(x_1,t) = 100 e^{-t}x_1^4(1-x_1)^4$ for $(x_1,t)\in[0,1]\times[0,1]$.
\label{example1}
\end{example}

\begin{table}[htbp]
\tabcolsep=9.0pt \label{table2}
\centering
		\caption{Comparisons for different $\alpha_1$ by one-sided and two-sided PGMRES solvers in one dimension.} 
		\begin{tabular}{c|c|c|ccccccc}
			\hline \multirow{2}{*}{$\alpha_1$} & \multirow{2}{*}{$M$} & \multirow{2}{*}{$N + 1$} & \multirow{2}{*}{Error} & \multicolumn{2}{c}{SPGM} & \multicolumn{2}{c}{OPGM} & \multicolumn{2}{c}{TPGM}\\
			\cline { 5 - 10 }  & &  & & CPU & Iter & CPU & Iter & CPU & \multicolumn{1}{c}{Iter} \\
			\hline 
			$1.30$& $2^{12}$ & $2^4$ & 9.44e-5 & 0.63 & 7.0 & 0.85 & 6.0 & 1.32 & 6.2 \\
			& & $2^5$ & 5.44e-6 & 0.79 & 8.0 & 1.02 & 6.6 & 1.51 & 6.8 \\
			& & $2^6$ & 3.19e-7 & 0.93 & 8.0 & 1.17 & 7.0 & 1.61 & 7.0  \\
			\cline{2-10}
			& $2^{14}$& $2^4$ & 6.64e-5 & 2.62 & 8.0 & 3.76 & 7.1 & 5.66 & 7.2 \\
			& & $2^5$ & 3.91e-6 & 3.36 & 8.8  & 4.35 & 7.2 & 6.62 & 7.5 \\
			& & $2^6$ & 2.37e-7 & 3.64 & 9.0 & 4.83 & 7.9 & 7.05 & 8.0 \\
			& & $2^7$ & 1.46e-8 & 6.73 & 9.1 & 6.98 & 8.0 & 9.85 & 8.0 \\
            \hline
                $1.50$ & $2^{12}$ & $2^4$ & 1.27e-4 & 0.62 & 7.0 & 0.82 & 6.0 & 1.27 & 6.0  \\
			& & $2^5$ & 7.30e-6 & 0.76 & 7.3 & 0.98 & 6.0 & 1.48 & 6.5 \\
			& & $2^6$ & 4.23e-7 & 0.85 & 8.0 & 1.07 & 6.8 & 1.60 & 7.0  \\
			\cline{2-10}
			& $2^{14}$& $2^4$ & 9.30e-5 & 2.59 & 8.0 & 3.58 & 6.4 & 5.32
			& 6.7  \\
			& & $2^5$ & 5.42e-6 & 3.22 & 8.2 & 4.22 & 7.0 & 6.33 & 7.0 \\
			& & $2^6$ & 3.20e-7 & 3.81 & 8.7 & 4.65 & 7.0 & 6.64 & 7.3  \\
			& & $2^7$ & 1.92e-8 & 7.33 & 9.0 & 6.63 & 7.3 & 9.54 & 7.8 \\
             \hline
                 $1.90$ & $2^{12}$ & $2^4$ & 2.34e-4 & 0.57 &  6.1
			& 0.73 & 5.0 & 1.17 & 5.0 \\
			& & $2^5$ & 1.44e-5 & 0.76 & 7.0 & 0.88 & 5.0 & 1.27 & 5.1\\
			& & $2^6$ & 8.81e-7 & 0.78 & 7.0 & 0.93 & 5.0 & 1.41 & 5.6 \\
			\cline{2-10}
			&$2^{14}$& $2^4$ & 1.78e-4 & 2.50 & 7.0 & 3.28 & 5.3 & 4.96 & 5.7 \\
			&& $2^5$ & 1.09e-5 & 2.97 & 7.2 & 3.61 & 5.4 & 5.52 & 5.9 \\
			&& $2^6$ & 6.69e-7 & 3.49 & 7.5 & 4.00 & 5.8 & 6.08 & 6.0 \\
			&& $2^7$ & 4.08e-8 & 6.21 & 7.9 & 5.74 & 6.0 & 8.51 & 6.0\\
   \hline
	\end{tabular}
\end{table}

Table \ref{table2} displays the errors, CPU times, and iterations obtained by SPGM, OPGM and TPGM, considering different values of $\alpha_1$, $M$, and $N$. The related results show both OPGM and TPGM have similar iteration numbers. Moreover, irrespective of variations in parameters $\alpha_1$, $M$, and $N$, OPGM requires fewer preconditioning steps per iteration compared to TPGM, which results in reduced CPU requirements reflected in the data and supported the theoretical results
presented in Theorem \ref{GMRES} and Theorem \ref{one-sidethm}.
\begin{figure}[htbp]\label{fig1}
	\centering
	\begin{minipage}{0.49\linewidth}
		\centering
\includegraphics[width=1.0\linewidth]{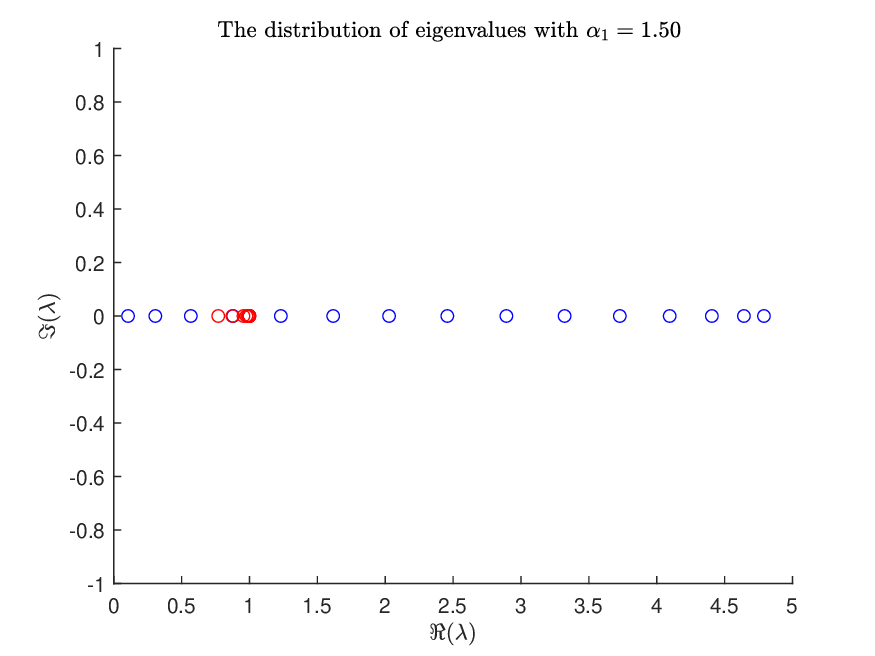}
	\end{minipage}
	\begin{minipage}{0.49\linewidth}
		\centering
\includegraphics[width=1.0\linewidth]{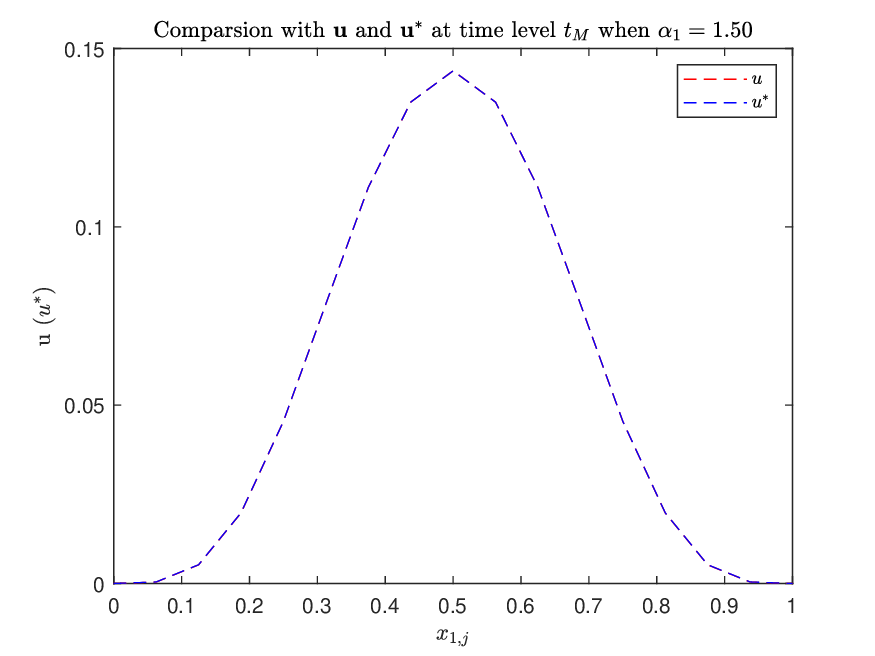}
	\end{minipage}
 \caption{The one-dimensional results when $N+1 = 2^4$ and $M=2^{12}$.}
\end{figure}

In Fig. \ref{fig1}, the right side presents a comparison between the exact solution and the numerical solution obtained in Example \ref{example1} at time layer $t_M$, with $\alpha_1=1.5$. Additionally, since $P_{\alpha}^{-1}\tilde{A}_{\alpha}^{(m)}$ is equal to $P_{\alpha}^{-1/2}\tilde{A}_{\alpha}^{(m)}P_{\alpha}^{-1/2}$, we only need to depict the eigenvalue distribution of $P_{\alpha}^{-1}\tilde{A}_{\alpha}^{(m)}$. The distribution of eigenvalues for both the coefficient matrix $\tilde{A}_{\alpha}^{(m)}$ and the preconditioned matrix $P_{\alpha}^{-1}\tilde{A}_{\alpha}^{(m)}$ is displayed. It can be seen in this figure that the eigenvalues of $\tilde{A}_{\alpha_1}^{(m)}$ denoted in color `blue' are more dispersed than those of the matrix $P_{\alpha_1}^{-1}\tilde{A}_{\alpha_1}^{(m)}$, indicated in color `red'. Notably, the eigenvalues of $P_{\alpha}^{-1}\tilde{A}_{\alpha}^{(m)}$ appear to be more tightly clustered compared to those of $\tilde{A}_{\alpha}^{(m)}$. This highlights the effectiveness and robustness of the PGMRES method discussed in Section \ref{Section:tau}.

\begin{example}
	In the two-dimensional grid of $(N+1) \times (N+1)$, instantaneous solute release point sources are set at the internal grid points. The initial concentration $\psi$ of these point sources satisfies $\psi=\psi(x_1,x_2)=10^4 x_1^4(1-x_1)^4 x_2^4(1-x_2)^4$, and the initial concentration of other boundary points is $0$. We consider diffusion coefficients $\kappa_1=\kappa_2=100$ and a variable coefficient $e(x_1,x_2,t)=(x_1^2 + x_2^2 + e^{-t})/100$. The exact solution is $u(x_1,x_2,t)=10^4 e^{-t} x_1^4(1-x_1)^4 x_2^4(1-x_2)^4$, $(x_1,x_2,t)\in[0,1]^2\times[0,1]$. 
 \label{example2}
\end{example}

\begin{table}[htbp]
\tabcolsep=7.0pt \label{table3}
\centering
		\caption{Comparisons for different $\alpha_1$ and $\alpha_2$ by one-sided and two-sided PGMRES solvers in two dimensions.}
		\begin{tabular}{c|c|c|ccccccc}
			\hline \multirow{2}{*}{$(\alpha_1,\ \alpha_2)$} & \multirow{2}{*}{$M$} & \multirow{2}{*}{$N + 1$} & \multirow{2}{*}{Error} & \multicolumn{2}{c}{SPGM} & \multicolumn{2}{c}{OPGM} & \multicolumn{2}{c}{TPGM}\\
			\cline { 5 - 10 }  & &  & & CPU & Iter & CPU & Iter & CPU & \multicolumn{1}{c}{Iter} \\
			\hline 
			$(1.10,\ 1.30)$ & $2^{12}$ & $2^4$ & 2.10e-5 & 8.25e+0 & 18.4 & 5.56e+0 & 8.0 & 7.52e+0 & 8.0  \\
                 & & $2^5$ & 1.23e-6 & 3.13e+1 & 22.0 & 1.47e+1 & 8.0 & 1.99e+1 & 8.6 \\
                 & & $2^6$ & 7.38e-8 & 8.89e+1 & 26.0 & 4.14e+1 & 8.5 &  5.09e+1 & 9.0 \\
		\cline{2-10}
                   & $2^{14}$& $2^4$ & 1.70e-5 & 3.53e+1 & 18.5 & 2.28e+1 & 8.2 & 3.18e+1 & 8.4 \\
                   & & $2^5$ & 1.01e-6 & 1.25e+2 & 22.1 & 6.23e+1 & 8.6 & 7.97e+1 & 8.8 \\
                   & & $2^6$ & 6.20e-8 & 3.71e+2 & 26.0 & 1.66e+2 & 9.0 & 2.24e+2 & 9.1 \\
                   & & $2^7$ & 3.89e-9 & 3.59e+3 & 31.0 & 8.54e+2 & 9.0 & 9.08e+2 & 9.4 \\
            \hline
                  $(1.30,\ 1.50)$ & $2^{12}$& $2^4$ & 2.80e-5 & 8.56e+0 & 19.0 & 5.25e+0 & 7.3 & 7.60e+0 & 8.0 \\
                   & & $2^5$ & 1.62e-6 & 3.43e+1 & 24.0 & 1.43e+1 & 8.0 & 1.89e+1 & 8.0\\
                   & & $2^6$ & 9.49e-8 & 9.61e+1 & 28.0 & 3.92e+1 & 8.0 & 5.27e+1 & 8.8 \\
		\cline{2-10}
                   & $2^{14}$& $2^4$ & 2.34e-5 & 3.62e+1 & 19.3 & 2.21e+1 & 7.7 &  2.97e+1  & 7.9 \\
                   && $2^5$ & 1.37e-6 & 1.35e+2 & 24.0 & 5.96e+1 & 8.0 &  7.83e+1 & 8.4  \\
                   && $2^6$ & 8.12e-8 & 3.94e+2 & 28.1 & 1.59e+2 & 8.2 & 2.09e+2 & 8.7 \\
                   && $2^7$ & 4.92e-9 & 3.70e+3 & 34.0 & 7.26e+2 & 8.6 & 9.12e+2 & 9.0 \\
            \hline
            $(1.70,\ 1.90)$ & $2^{12}$& $2^4$ & 5.00e-5 & 9.40e+0 & 21.0 &  4.66e+0 & 6.0 & 6.39e+0 & 6.0\\
                   && $2^5$ & 3.03e-6 & 3.64e+1 & 26.0 & 1.25e+1 & 6.0 & 1.70e+1 & 7.0\\
                   && $2^6$ &  1.83e-7 & 1.11e+2 & 32.0 & 3.20e+1 & 6.0 & 4.86e+1 & 8.0 \\
		\cline{2-10}
                   &$2^{14}$& $2^4$ & 4.38e-5 & 3.91e+1 & 21.1 & 1.94e+1 & 6.3 & 2.70e+1 & 6.7\\
                   && $2^5$ & 2.66e-6 & 1.51e+2 & 26.9 & 5.18e+1 & 6.6 & 7.03e+1 & 7.0\\
                   && $2^6$ & 1.61e-7 & 4.34e+2 & 32.0 & 1.60e+2 & 6.9 & 1.86e+2 & 7.0\\
                    && $2^7$ & 9.69e-9 & 3.25e+3 & 40.6 & 8.45e+2 & 7.0 & 1.31e+3 & 7.0\\
   \hline
	\end{tabular}
\end{table}

\begin{figure}[htbp]\label{fig2}
	\centering
	\begin{minipage}{0.49\linewidth}
		\centering
		\includegraphics[width=1.0\linewidth]{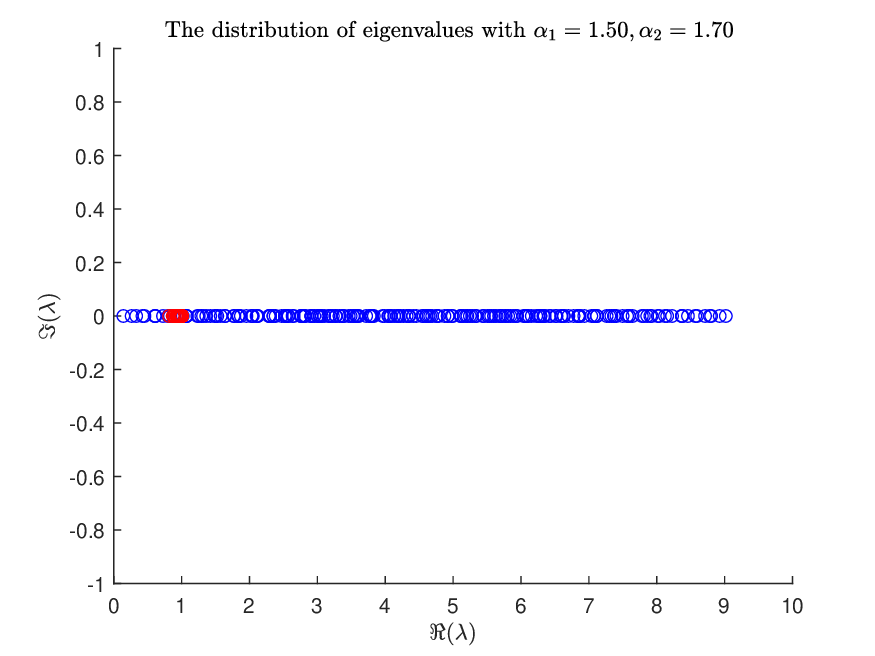}
	\end{minipage}
	\begin{minipage}{0.49\linewidth}
		\centering
		\includegraphics[width=1.0\linewidth]{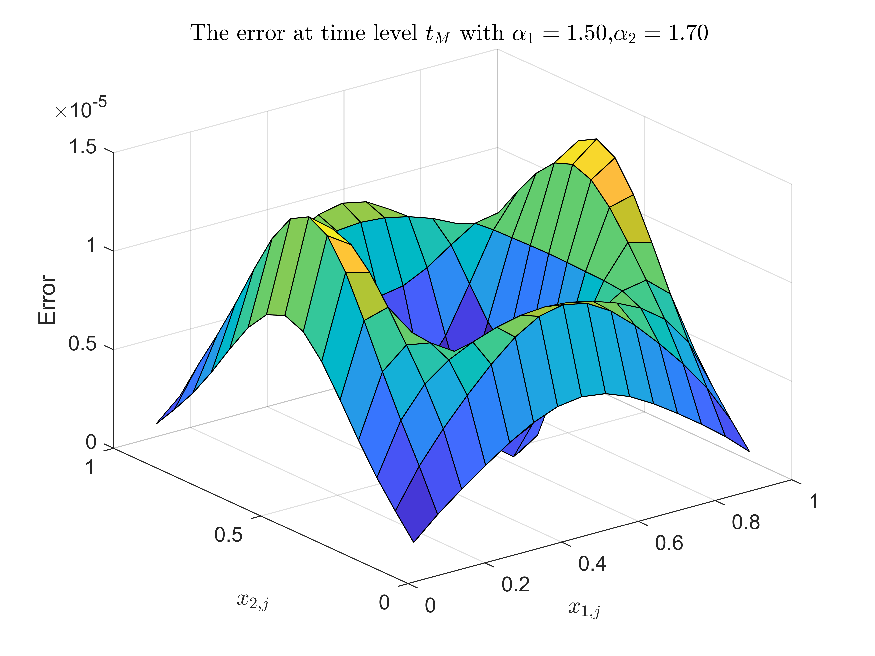}
	\end{minipage}
 \caption{The two-dimensional  results when $N+1 = 2^4$ and $M=2^{12}$.}
\end{figure}

In Example \ref{example2} of the two-dimensional problem, our computed results agree closely with those shown in Example \ref{example1}, as demonstrated in Table \ref{table3}. The results of Table \ref{table3} reveal that CPU time and the number of iterations of TPGM perform well, echoing the findings in Table \ref{table2}. Furthermore, the left side of Fig. \ref{fig2} presents a comparison of eigenvalue distributions between the matrices $\tilde{A}_{\alpha}^{(m)}$ and $P_{\alpha}^{-1}\tilde{A}_{\alpha}^{(m)}$, confirming the conclusions drawn in Fig. \ref{fig1}. The figure on the right depicts the error change at time layer $t_M$ when $\alpha_1=1.50$ and $\alpha_2=1.70$, effectively showing the strong correspondence between the exact solution and the numerical solution, which presents the benefit of the preconditioning technology we designed. Given the significant contrast between the results of SPGM and OPGM in Tables \ref{table2}-\ref{table3}, OPGM has proven to be the optimal choice for these cases. However, we will not reiterate this conclusion in the Example \ref{example3}.

\begin{example}
	In the three-dimensional grid of $(N+1) \times (N+1) \times (N+1)$, instantaneous solute release point sources are set at the internal grid points. The initial concentration $\psi$ of these point sources satisfies $\psi=\psi(x_1,x_2,x_3)=10^8 x_1^4(1-x_1)^4 x_2^4(1-x_2)^4 x_3^4(1-x_3)^4$, and the initial concentration of other boundary points is $0$. We consider the diffusion coefficients $\kappa_1=100,\kappa_2=85,\kappa_3=103$ and a variable coefficient $e(x_1,x_2,x_3,t)=(x_1^2 + x_2^2 + x_3^2+ e^{-t})/100$. The exact solution is $u(x_1,x_2,x_3,t)=10^8 e^{-t} x_1^4(1-x_1)^4 x_2^4(1-x_2)^4 x_2^4(1-x_3)^4$, $(x_1,x_2,x_3,t)\in[0,1]^3\times[0,1]$.   
 \label{example3}
\end{example}

\begin{table}[htbp]
\tabcolsep=9.0pt
\centering
		\caption{Comparisons for different $\alpha_1$, $\alpha_2$ and $\alpha_3$ by one-sided and two-sided PGMRES solvers in three dimensions.} 
		\begin{tabular}{c|c|c|ccccc}
			\hline \multirow{2}{*}{$(\alpha_1,\ \alpha_2,\ \alpha_3)$} & \multirow{2}{*}{$M$} & \multirow{2}{*}{$N + 1$} & \multirow{2}{*}{Error} & \multicolumn{2}{c}{OPGM} & \multicolumn{2}{c}{TPGM}\\
			\cline { 5 - 8 }  & &  & & CPU & Iter & CPU & \multicolumn{1}{c}{Iter} \\
			\hline 
			$(1.10,\ 1.30,\ 1.50)$ & $2^{10}$ & $2^3$ & 1.28e-4 & 8.59e+0 & 7.6 & 9.78e+0 & 8.0\\
		&& $2^4$ & 8.25e-6 & 3.50e+1 & 8.0 & 3.95e+1 & 8.0 \\
		&& $2^5$ & 5.18e-7 & 1.84e+2 & 9.0 & 2.35e+2 & 9.0  \\
		\cline{2-8}
		&$2^{12}$& $2^3$ & 1.23e-4 & 3.38e+1 & 8.0 & 3.86e+1 & 8.0  \\
		&& $2^4$ & 8.04e-6 & 1.46e+2 & 8.3 & 1.67e+2 & 8.7 \\
		&& $2^5$ & 5.06e-7 & 8.90e+2 & 9.0 & 9.28e+2 & 9.0  \\
               \hline
            $(1.30,\ 1.50,\ 1.70)$ & $2^{10}$& $2^3$ & 1.49e-4 & 7.91e+0 & 7.0 & 9.00e+0 & 7.0 \\
		&& $2^4$ & 9.31e-6 & 3.43e+1 & 8.0 & 3.70e+1 & 8.0 \\
		&& $2^5$ & 5.81e-7 & 1.86e+2 & 8.0 & 2.07e+2 & 8.0  \\
		\cline{2-8}
		&$2^{12}$& $2^3$ & 1.45e-4 & 3.21e+1 & 7.0 & 3.60e+1 & 7.0  \\
		&& $2^4$ & 9.15e-6 & 1.41e+2 & 8.0 & 1.56e+2 & 8.0 \\
		&& $2^5$ &  5.73e-7 & 7.44e+2 & 8.0 & 7.77e+2 & 8.2  \\
           \hline
            $(1.50,\ 1.70,\ 1.90)$ & $2^{10}$& $2^3$ & 1.70e-4 & 7.07e+0 & 6.0 & 7.96e+0 & 6.0\\
		&& $2^4$ & 1.02e-5 & 3.13e+1 & 7.0 & 3.61e+1 & 7.0 \\
		&& $2^5$ & 6.34e-7 & 1.69e+2 & 7.0 & 1.93e+2 & 7.0 \\
		\cline{2-8}
		&$2^{12}$& $2^3$ & 1.67e-4 & 2.86e+1 & 6.1 & 3.46e+1 & 6.5 \\
		&& $2^4$ & 1.01e-5 & 1.29e+2 & 7.0 & 1.44e+2 & 7.0 \\
		&& $2^5$ & 6.29e-7 & 6.73e+2 & 7.0 & 9.07e+2 & 7.9 \\
   \hline
	\end{tabular}
\label{table4}
\end{table}
\begin{figure}[htbp]\label{fig3}
	\centering
	\begin{minipage}{0.49\linewidth}
		\centering
		\includegraphics[width=1.0\linewidth]{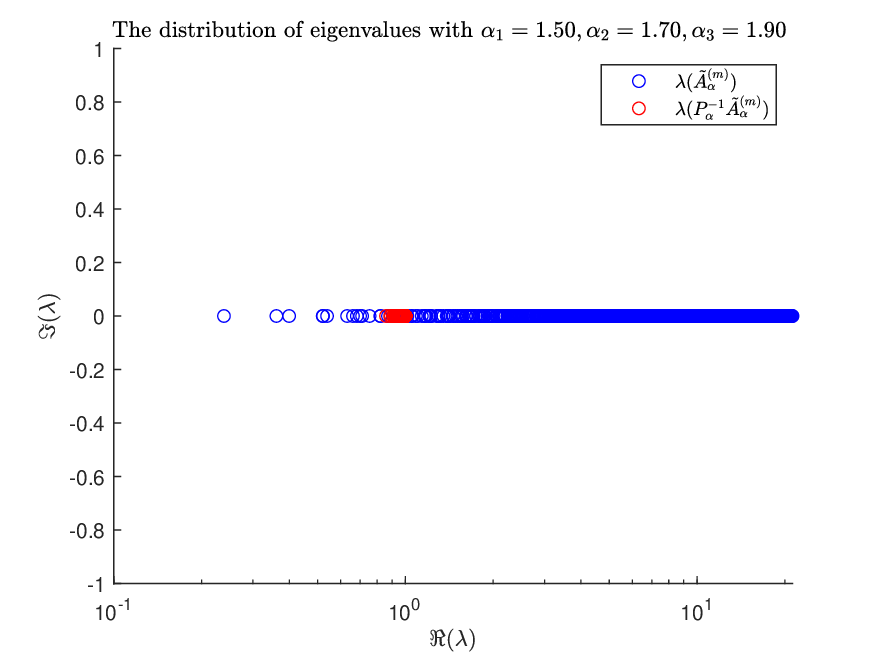}
	\end{minipage}
	\begin{minipage}{0.49\linewidth}
		\centering
		\includegraphics[width=1.0\linewidth]{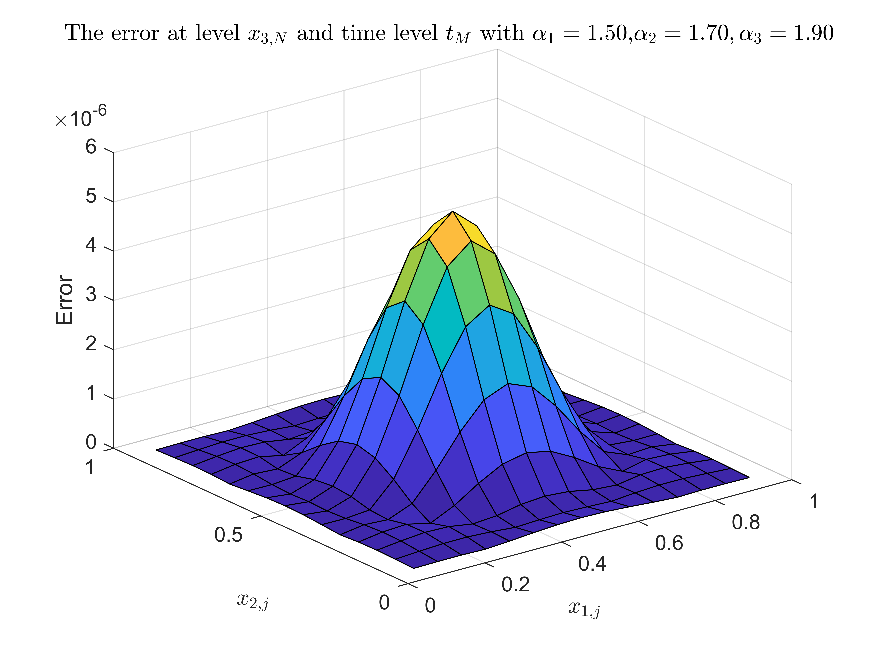}
	\end{minipage}
 \caption{ The three-dimensional  results when $N+1 = 2^4$ and $M=2^{12}$.}
\end{figure}

According to Table \ref{table4} and Fig. \ref{fig3}, the results from Example \ref{example3} align closely with the findings of the first two examples, and we do not need to discuss them in more detail here. However, during the construction of the 3D figure (i.e., Fig. \ref{fig3}), we observed roughness at the boundaries, particularly becoming apparent with increasing dimensions. The described phenomenon already occurs and is discussed in the numerical tests of the work \cite{BU2014}. Although this phenomenon is currently acknowledged by all, we aim to address this challenge in our future work with the help of numerical discretizations on locally refined meshes.

\section{Conclusions}
\label{sec-end}
We employ the quasi-compact FCD scheme to discretize the $d$-dimensional RSFDE with variable coefficients and develop a novel algorithm utilizing the characteristics of the coefficient matrix. For the $d$-dimensional problem, we introduce a preconditioner based on the sine transform. Theoretical analysis confirms that the upper bound of the relative residual norm using the GMRES method with our proposed preconditioner is independent of the grid size and dimension. As a result, the PGMRES method achieves linear convergence. Numerical simulations prove its efficiency and demonstrate its advantages. 

As a further step, we could consider a precise analysis of the distributional spectral and singular value results of the underlying matrix-sequences and this could benefit from the tools of generalized locally Toeplitz $*$-algebras of matrix-sequences \cite{GLT-I-book,GLT-II-book}. Moreover, for high-dimensional models, i.e., $d\geq 4$, we will consider to use Tucker decomposition to accelerate the solutions of the discretized system in tensor format, in combination with multi-iterative solvers when the variable coefficients are highly oscillating or with jumps.
%
\section*{Appendix A}
\label{appA}
In the following, for problems in $d$ dimensions (where $d=1, 2, 3$), regarding the vector $\mathbf{F}^{(m)}$ in \eqref{eq2.1}, its definition varies according to the dimensionality of the problem.
\begin{equation*}
\begin{aligned}
&\text{One-dimensional case:}\ \mathbf{F}^{(m)} = B_{\alpha}\mathbf{f}^{(m)} + \dfrac{\alpha_1}{24}\mathbf{f}_0^{(m)}.\\
&\text{Two-dimensional case:}\ \mathbf{F}^{(m)} = B_{\alpha}\mathbf{f}^{(m)}
+ \left(B_{\alpha_2}\otimes \frac{\alpha_1 I_{N}}{24}\right) \mathbf{f}_1^{(m)}+\frac{\alpha_2}{24}
\begin{bmatrix}
              B_{\alpha_1}\mathbf{g}_1^{(m)}\\
              \mathbf{0}\\
              B_{\alpha_1} \mathbf{g}_2^{(m)}
        \end{bmatrix}
+\frac{\alpha_1 \alpha_2}{24^2}
 \begin{bmatrix}
              \mathbf{g}_3^{(m)}\\
              \mathbf{0}\\
             \mathbf{g}_4^{(m)}
\end{bmatrix}. \\
&\text{Three-dimensional case:}\ \mathbf{F}^{(m)}= B_{\alpha}\mathbf{f}^{(m+1/2)} + \left(B_{\alpha_3} \otimes B_{\alpha_2} \otimes \frac{\alpha_1 I_{N}}{24}\right) \mathbf{f}_2^{(m)} + \left(B_{\alpha_3} \otimes \frac{\alpha_2 I_{N}}{24} \otimes B_{\alpha_1}\right) \mathbf{f}_3^{(m)}\\
&  + \left(B_{\alpha_3} \otimes \frac{\alpha_2 I_{N}}{24} \otimes \frac{\alpha_1 I_{N}}{24}\right) \mathbf{f}_4^{(m)}
+ \frac{\alpha_3}{24}
\begin{bmatrix}
              (B_{\alpha_2} \otimes B_{\alpha_1})\mathbf{g}_5^{(m)}\\
              \mathbf{0}\\
              (B_{\alpha_2} \otimes B_{\alpha_1}) \mathbf{g}_6^{(m)}\\
        \end{bmatrix} \\
&+\frac{\alpha_3}{24}
\begin{bmatrix}
              (B_{\alpha_2} \otimes \frac{\alpha_1 I_{N}}{24})\mathbf{g}_7^{(m)}\\
              \mathbf{0}\\
              (B_{\alpha_2} \otimes \frac{\alpha_1 I_{N}}{24}) \mathbf{g}_8^{(m)}\\
        \end{bmatrix} +\frac{\alpha_3}{24}
\begin{bmatrix}
              (\frac{\alpha_2 I_{N}}{24} \otimes B_{\alpha_1})\mathbf{g}_{9}^{(m)}\\
              \mathbf{0}\\
              (\frac{\alpha_2 I_{N}}{24} \otimes B_{\alpha_1}) \mathbf{g}_{10}^{(m)}\\
        \end{bmatrix} +\frac{\alpha_3\alpha_2\alpha_1}{24^3}{\mathbf{f}_5^{(m)}}.
        %
\end{aligned} 
\end{equation*} 
%
%
      %
In the above vector $\mathbf{F}^{(m)}$, the definitions of the relevant vectors are as follows.
\begin{equation*}
\begin{aligned}
\mathbf{f}^{(m)} =& f(\mathbb{G},t_{m}) \\
\mathbf{f}_0^{(m)}=&[f(L_1,t_{m}),\underbrace{0,\ldots,0}_{N-2},f(R_1,t_{m})]^\top,\\
\mathbf{f}_1^{(m)}=&\big{[}f(x_{1,0},x_{2,1},t_m), \underbrace{0,\cdots,0}_{N-2},f(x_{1,N+1},x_{2,1},t_m),
f(x_{1,0},x_{2,2},t_m),\underbrace{0,\cdots,0}_{N-2},f(x_{1,N+1},x_{2,2},t_m),\\
& \cdots,
f(x_{1,0},x_{2,N},t_m),\underbrace{0,\cdots,0}_{N-2}, f(x_{1,N+1},x_{2,N},t_m)\big{]}^{\top},\\
\mathbf{f}_2^{(m)}=&\big{[}f(x_{1,0},x_{2,1},x_{3,1},t_m), \underbrace{0,\cdots,0}_{N-2},f(x_{1,N+1},x_{2,1},x_{3,1},t_m),
f(x_{1,0},x_{2,2},x_{3,1},t_m),\underbrace{0,\cdots,0}_{N-2},\\
& f(x_{1,N+1},x_{2,2},x_{3,1},t_m), \cdots,f(x_{1,0},x_{2,N},x_{3,1},t_m),\underbrace{0,\cdots,0}_{N-2},
f(x_{1,N+1},x_{2,N},x_{3,1},t_m),\\
&\cdots,f(x_{1,0},x_{2,1},x_{3,N},t_m), \underbrace{0,\cdots,0}_{N-2},f(x_{1,N+1},x_{2,1},x_{3,N},t_m),
f(x_{1,0},x_{2,2},x_{3,N},t_m),\underbrace{0,\cdots,0}_{N-2},\\
& f(x_{1,N+1},x_{2,2},x_{3,N},t_m), \cdots,f(x_{1,0},x_{2,N},x_{3,N},t_m),\underbrace{0,\cdots,0}_{N-2},
f(x_{1,N+1},x_{2,N},x_{3,N},t_m)\big{]}^{\top},\\
\end{aligned} 
\end{equation*} 
\begin{equation*}
\begin{aligned}
\mathbf{f}_3^{(m)}=&\big{[}f(x_{1,1},x_{2,0},x_{3,1},t_m),\cdots,f(x_{1,N},x_{2,0},x_{3,1},t_m),
\underbrace{0,\cdots,0}_{(N-2)N},f(x_{1,1},x_{2,N+1},x_{3,1},t_m),\cdots,\\
&f(x_{1,N},x_{2,N+1},x_{3,1},t_m),f(x_{1,1},x_{2,0},x_{3,2},t_m),\cdots,f(x_{1,N},x_{2,0},x_{3,2},t_m),
\underbrace{0,\cdots,0}_{(N-2)N},\\
&f(x_{1,1},x_{2,N+1},x_{3,2},t_m),\cdots,f(x_{1,N},x_{2,N+1},x_{3,2},t_m),\cdots,f(x_{1,1},x_{2,0},x_{3,N},t_m),\\
&\cdots,f(x_{1,N},x_{2,0},x_{3,N},t_m),
\underbrace{0,\cdots,0}_{(N-2)N},f(x_{1,1},x_{2,N+1},x_{3,N}, t_m),\cdots, f(x_{1,N},x_{2,N + 1},x_{3,N}, t_m)\big{]}^{\top},\\
\mathbf{f}_4^{(m)}=&\big{[}f(x_{1,0},x_{2,0},x_{3,1},t_m),\underbrace{0,\cdots,0}_{N-2},f(x_{1,N+1},x_{2,0},x_{3,1},t_m),
\underbrace{0,\cdots,0}_{(N-2)N},f(x_{1,1},x_{2,N+1},x_{3,1},t_m),\\
&\underbrace{0,\cdots,0}_{N-2}, f(x_{1,N+1},x_{2,N+1},x_{3,1},t_m),\cdots,f(x_{1,0},x_{2,0},x_{3,N},t_m),\underbrace{0,\cdots,0}_{N-2},f(x_{1,N+1},x_{2,0},x_{3,N},t_m),\\
&\underbrace{0,\cdots,0}_{(N-2)N},f(x_{1,1},x_{2,N+1},x_{3,N},t_m),\underbrace{0,\cdots,0}_{N-2}, f(x_{1,N+1},x_{2,N+1},x_{3,N},t_m),\big{]}^{\top},\\
\mathbf{f}_{5}^{(m)}=&\big{[}f(x_{1,0},x_{2,0},x_{3,0},t_m),\underbrace{0,\cdots,0}_{N-2},f(x_{1,N+1},x_{2,0},x_{3,0},t_m),\underbrace{0,\cdots,0}_{(N-2)N},
f(x_{1,0},x_{2,N+1},x_{3,0},t_m),\underbrace{0,\cdots,0}_{N-2},\\
&f(x_{1,N+1},x_{2,N+1},x_{3,0},t_m),\underbrace{0,\cdots,0}_{(N-2)N^2},f(x_{1,0},x_{2,0},x_{3,N+1},t_m),\underbrace{0,\cdots,0}_{N-2},f(x_{1,N+1},x_{2,0},x_{3,N+1},t_m),\\
&\underbrace{0,\cdots,0}_{(N-2)N},f(x_{1,0},x_{2,N+1},x_{3,N+1},t_m),\underbrace{0,\cdots,0}_{N-2},f(x_{1,N+1},x_{2,N+1},x_{3,N+1},t_m)\big{]}^{\top}
\end{aligned}
\end{equation*}

and
\begin{equation*}
\begin{aligned}
\mathbf{g}_1^{(m)}=&\big{[}f(x_{1,1},x_{2,0},t_m),f(x_{1,2},x_{2,0},t_m),\cdots,f(x_{1,N},x_{2,0},t_m)\big{]}^{\top},\\
\mathbf{g}_2^{(m)}=&\big{[}f(x_{1,1},x_{2,N+1},t_m),f(x_{1,2},x_{2,N+1},t_m),\cdots,f(x_{1,N},x_{2,N+1},t_m)\big{]}^{\top},\\
\mathbf{g}_3^{(m)}=&\big{[}f(x_{1,0},x_{2,0},t_m),\underbrace{0,\cdots,0}_{N-2},f(x_{1,N+1},x_{2,0},t_m)\big{]}^{\top},\\
\mathbf{g}_4^{(m)}=&\big{[}f(x_{1,0},x_{2,N+1},t_m),\underbrace{0,\cdots,0}_{N-2},f(x_{1,N+1},x_{2,N+1},t_m)\big{]}^{\top},\\
\mathbf{g}_5^{(m)}=&\big{[}f(x_{1,1},x_{2,1},x_{3,0},t_m),\cdots,f(x_{1,N},x_{2,1},x_{3,0},t_m),f(x_{1,1},x_{2,2},x_{3,0},t_m),\cdots,f(x_{1,N},x_{2,2},x_{3,0},t_m),\\
&\cdots,f(x_{1,1},x_{2,N},x_{3,0},t_m),\cdots,f(x_{1,N},x_{2,N},x_{3,0},t_m)\big{]}^{\top},\\
\mathbf{g}_6^{(m)}=&\big{[}f(x_{1,1},x_{2,1},x_{3,N+1},t_m),\cdots,f(x_{1,N},x_{2,1},x_{3,N+1},t_m),f(x_{1,1},x_{2,2},x_{3,N+1},t_m),\cdots,\\
&f(x_{1,N},x_{2,2},x_{3,N+1},t_m),\cdots,f(x_{1,1},x_{2,N},x_{3,N+1},t_m),\cdots,f(x_{1,N},x_{2,N},x_{3,N+1},t_m)\big{]}^{\top},\\
\mathbf{g}_7^{(m)}=&\big{[}f(x_{1,0},x_{2,1},x_{3,0},t_m),\underbrace{0,\cdots,0}_{N-2},f(x_{1,N+1},x_{2,1},x_{3,0},t_m),f(x_{1,0},x_{2,2},x_{3,0},t_m),\underbrace{0,\cdots,0}_{N-2},\\
&f(x_{1,N+1},x_{2,2},x_{3,0},t_m),\cdots, f(x_{1,0},x_{2,N},x_{3,0},t_m),\underbrace{0,\cdots,0}_{N-2},f(x_{1,N+1},x_{2,N},x_{3,0},t_m)\big{]}^{\top},
\end{aligned}
\end{equation*}
\begin{equation*}
\begin{aligned}
\mathbf{g}_8^{(m)}=&\big{[}f(x_{1,0},x_{2,1},x_{3,N+1},t_m),\underbrace{0,\cdots,0}_{N-2},f(x_{1,N+1},x_{2,1},x_{3,N+1},t_m),f(x_{1,0},x_{2,2},x_{3,N+1},t_m),\underbrace{0,\cdots,0}_{N-2}, \\
&f(x_{1,N+1},x_{2,2},x_{3,N+1},t_m),\cdots,
f(x_{1,0},x_{2,N},x_{3,N+1},t_m),\underbrace{0,\cdots,0}_{N-2},f(x_{1,N+1},x_{2,N},x_{3,N+1},t_m)\big{]}^{\top},\\
\mathbf{g}_{9}^{(m)}=&\big{[}f(x_{1,1},x_{2,0},x_{3,0},t_m),\cdots,f(x_{1,N},x_{2,0},x_{3,0},t_m),\underbrace{0,\cdots,0}_{(N-2)N},
f(x_{1,1},x_{2,N+1},x_{3,0},t_m),\cdots, \\
&f(x_{1,N},x_{2,N+1},x_{3,0},t_m)\big{]}^{\top},\\
\mathbf{g}_{10}^{(m)}=&\big{[}f(x_{1,1},x_{2,0},x_{3,N+1},t_m),\cdots,f(x_{1,N},x_{2,0},x_{3,N+1},t_m),\underbrace{0,\cdots,0}_{(N-2)N},
f(x_{1,1},x_{2,N+1},x_{3,N+1},t_m),\\
&\cdots,f(x_{1,N},x_{2,N+1},x_{3,N+1},t_m)\big{]}^{\top}.\\
\end{aligned}
\end{equation*}

\section*{Appendix B}\label{appB}

At the discrete point $(x_{1,j},t_{m+1/2})$, the one-dimensional case of Eq. \eqref{RSFDE} is as follows:
\begin{equation}\label{eqB1}\tag{B1}
    e(x_{1,j_1},t_{m+1/2}) \partial_t u(x_{1,j_1},t_{m+1/2}) = \kappa_1 \partial^{\alpha_1}_{x_1} u(x_{1,j_1},t_{m+1/2})+f(x_{1,j_1},t_{m+1/2}),\ j_1 \in \Upsilon,\ 0 \leq m \leq M-1.
\end{equation}

 Taylor expansion gives
\begin{equation}\label{eqB2}\tag{B2}
e_{j_1}^{m+1/2} \delta_t U_{j_1}^{m+1/2} = \frac{1}{2}\kappa_1 \partial_{x_1}^{\alpha_1} \left(U_{j_1}^{m+1} +  U_{j_1}^{m}\right) + f_{j_1}^{m+1/2} + R_{j_1}^m,\ j \in \Upsilon,\ 0 \leq m \leq M-1,
\end{equation}
where 
\begin{equation}\notag
\delta_t U_j^{m+1/2}=\frac{u\left(x_{1, j}, t_{m+1}\right)-u\left(x_{1, j}, t_{m}\right)}{\Delta t}+\mathcal{O}\left(\Delta t^2\right),
\end{equation}
and truncation error $R_{j_1}^m$ satisfies 
$\left|R_{j_1}^m\right| \leqslant c\left(\Delta t^2\right)$ for a positive constant $c$ independent of $\Delta t$ and $h_1$.

Applying the compact operator $\mathscr{H}_{\alpha_1}$ to both sides of Eq. \eqref{eqB2}, we obtain
\begin{equation}\label{eqB3}\tag{B3}
\mathscr{H}_{\alpha_1}\left(e_{j_1}^{m+1/2} \delta_t U_{j_1}^{m+1/2}\right) = \frac{1}{2}\kappa_1 \delta_{x_1}^{\alpha_1} \left(U_{j_1}^{m+1} +  U_{j_1}^{m}\right) + \mathscr{H}_{\alpha_1}f_{j_1}^{m+1/2} + Q_{j_1}^m,
\end{equation}
where 
\begin{equation}
\begin{aligned}\notag
\mathscr{H}_{\alpha_1}\left(e_{j_1}^{m+1/2} \delta_t U_{j_1}^{m+1/2}\right) &= \frac{\alpha_1}{24} e_{j_1-1}^{m+1/2} \delta_t U_{j_1-1}^{m+1/2}+\left(1-\frac{\alpha_1}{12}\right) e_{j_1}^{m+1/2} \delta_t U_{j_1}^{m+1/2}+\frac{\alpha_1}{24} e_{j_1+1}^{m+1/2} \delta_t U_{j_1+1}^{m+1/2},\\
\mathscr{H}_{\alpha_1}f_{j_1}^{m+1/2} &= \frac{\alpha_1}{24} f_{j_1-1}^{m+1/2}+\left(1-\frac{\alpha_1}{12}\right) f_{j_1}^{m+1/2}+\frac{\alpha_1}{24} f_{j_1+1}^{m+1/2},\\
\delta_{x_1}^{\alpha_1} U_{j_1}^{m+1} &= \mathscr{H}_{\alpha_1}\partial^{\alpha_1}_{x_1} U_{j_1}^{m+1}+\mathcal{O}\left(h_1^4\right)
\end{aligned}  
\end{equation}
with $\left|Q_{j_1}^m\right| \leqslant c\left(\Delta t^2+h_1^4\right)$.

Eliminating the small terms $Q_{j_1}^m$ and replacing the exact solution $U_{j_1}^m$ with its numerical approximation $u_{j_1}^m$, we obtain the quasi-compact difference scheme:
\begin{equation}\label{eqB4}\tag{B4}
\mathscr{H}_{\alpha_1}\left(e_{j_1}^{m+1/2} \delta_t u_{j_1}^{m+1/2}\right) = \frac{1}{2}\kappa_1 \delta_{x_1}^{\alpha_1} \left(u_{j_1}^{m+1} +  u_{j_1}^{m}\right) + \mathscr{H}_{\alpha_1}f_{j_1}^{m+1/2} \ (\text{one-dimensional case of Eq. \eqref{eq2.2}}).
\end{equation}

From this discrete scheme, the corresponding matrix-vector form is immediately:
\begin{equation*}
\begin{aligned}
&\left[\begin{array}{ccccc}
1-\frac{\alpha_1}{12} & \frac{\alpha_1}{24} & & & \\
\frac{\alpha_1}{24} & 1-\frac{\alpha_1}{1 2} & \frac{\alpha_1}{24} & & \\
& \ddots & \ddots & \ddots \\
& & \frac{\alpha_1}{24} & 1-\frac{\alpha_1}{12} & \frac{\alpha_1}{24} \\
& & & \frac{\alpha_1}{2 4} & 1-\frac{\alpha_1}{12}
\end{array}\right]\left[\begin{array}{cccc}
e_1^{m+1/2} & & & \\
& e_2^{m+1/2} & & \\
& & \ddots & \\
& & & e_N^{m+1/2}
\end{array}\right]\left(\mathbf{u}^{(m+1)}-\mathbf{u}^{(m)}\right)=\\
&\frac{-\kappa_1 \Delta t}{2 h_1^{\alpha_1}} \left[\begin{array}{ccccc}
s_0^{\left(\alpha_1\right)} & s_1^{\left(\alpha_1\right)} & \ldots & s_{N-2}^{\left(\alpha_1\right)} & s_{N-1}^{\left(\alpha_1\right)} \\
s_1^{\left(\alpha_1\right)} & s_0^{\left(\alpha_1\right)} & s_1^{\left(\alpha_1\right)} & \ldots & s_{N-2}^{\left(\alpha_1\right)} \\
\vdots & \ddots & \ddots & \ddots & \vdots \\
s_{N-2}^{\left(\alpha_1\right)} & \ddots & \ddots & \ddots & s_1^{\left(\alpha_1\right)} \\
s_{N-1}^{\left(\alpha_1\right)} & s_{N-2}^{\left(\alpha_1\right)} & \ldots & s_1^{\left(\alpha_1\right)} & s_0^{\left(\alpha_1\right)}
\end{array}\right]\left(\mathbf{u}^{(m+1)}+\mathbf{u}^{(m)}\right) + \Delta t \mathbf{F}^{(m+1/2)},\\ &(\mathbf{F}^{(m+1/2)}\ \text{given by \nameref{appA})}
\end{aligned}
\end{equation*}
which agrees with Eq. \eqref{tls}.

\section*{Acknowledgments}
The authors would like to thank Dr. Xin Huang for her help to improve our code of Example 4.3. X.-M. Gu wants to especially thank Prof. Cornelis W. Oosterlee at Utrecht University for reading the draft of our manuscript and providing many helpful suggestions. 
\section*{Funding}
This work of Z.-H. She is supported by the National Natural Science Foundation of China (Grant No. 12301481), Guangdong Province Humanities and Social Sciences Planning Project (Grant No. GD25YYJ39). The work of X.-M. Gu is supported by the National Natural Science Foundation of China (Grant No. 72431008). The work of S. Serra-Capizzano is funded from the European High-Performance Computing Joint Undertaking (JU) under grant agreement No. 955701. The JU receives support from the European Union’s Horizon 2020 research and innovation programme and Belgium, France, Germany, Switzerland.
Furthermore S. Serra-Capizzano is grateful for the support of the Italian GNCS agency and of the Laboratory of Theory, Economics and Systems – Department of Computer Science at Athens University of Economics and Business.
\section*{Contributions}
Conceptualization: ZHS, XZ, XMG. Formal Analysis: ZHS, XZ, XMG, SSC. Methodology: ZHS, XZ, XMG, SSC. Investigation: ZHS, XZ, XMG, SSC. Software: ZHS, XZ, XMG. Writing: ZHS, XZ, XMG. Review: ZHS, XMG, SSC. Project administration: ZHS, XMG, SSC.
\section*{Declarations}
\subsection*{Conflict of interest}
The authors declare that they have no conflict of interest.

\bibliographystyle{siamplain}
\bibliography{refX}
\end{document}